\documentclass[11pt]{article}

\usepackage[T1]{fontenc}
\usepackage[utf8]{inputenc}
\usepackage[british]{babel}
\usepackage[
rm={oldstyle=false,proportional=true},
sf={oldstyle=false,proportional=true},
tt={oldstyle=false,proportional=true}
]{cfr-lm}
\usepackage[babel]{microtype}

\usepackage{authblk}

\usepackage{mathrsfs}
\usepackage{mathtools}
\usepackage{amsfonts}
\usepackage{amssymb}
\usepackage{amsthm}

\usepackage{enumerate}

\usepackage{tikz}
\usetikzlibrary{decorations.pathreplacing}

\usepackage{xcolor}
\usepackage{array}
\usepackage{colortbl}
\usepackage{booktabs}
\usepackage{multirow}

\usepackage[misc]{ifsym}

\usepackage[noadjust]{cite}
\usepackage[hidelinks,bookmarksnumbered]{hyperref}
\usepackage[capitalise,noabbrev]{cleveref}

\usepackage{orcidlink}

\newtheorem{theorem}{Theorem}[section]
\newtheorem{lemma}[theorem]{Lemma}
\newtheorem{proposition}[theorem]{Proposition}
\newtheorem{corollary}[theorem]{Corollary}

\theoremstyle{definition}
\newtheorem{definition}[theorem]{Definition}
\newtheorem{problem}[theorem]{Open problem}

\DeclareMathOperator{\outcome}{o}
\DeclareMathOperator{\outcomeL}{o_L}
\DeclareMathOperator{\outcomeR}{o_R}
\DeclareMathOperator{\birth}{\tilde{b}}
\newcommand{\set}[1]{\mathcal{#1}}

\NewDocumentCommand\pf{smO{}}{%
  \ensuremath{
    \operatorname{pf}_{%
      \IfBooleanTF{#1}{#3}{\set{#3}}%
    }(%
      \IfBooleanTF{#1}{#2}{\set{#2}}%
    )
  }%
}

\DeclareMathOperator{\maug}{\set{M}_\text{aug}}
\DeclareMathOperator{\ltp}{l}
\DeclareMathOperator{\ntp}{n}
\DeclareMathOperator{\rtp}{r}

\makeatletter
\newcommand{\tomb}{
    \begin{tikzpicture}[scale=0.11]
        \filldraw (0,0) rectangle (1,2);
    \end{tikzpicture}
    \@ifnextchar,{\hspace{0.1em}}{}%
}
\makeatother

\newcounter{authcount}

\NewDocumentCommand{\authDetails}{m m m o}{%
    \stepcounter{authcount}%
    \IfNoValueTF{#4}{%
        \author[\arabic{authcount}]%
    }%
    {%
        \author[#4]%
    }%
    {%
        \mbox{#1\,$^{\textrm{\href{mailto:#2}{\Letter}}\,\,%
        \ifx&#3&\else\raisebox{-0.2ex}{\orcidlink{#3}}\,\fi}$}%
    }%
}

\title{On sums of $\mathscr{P}$-free forms under mis\`ere play}

\authDetails{Alfie Davies}{research@alfied.xyz}{0000-0002-4215-7343}[1]
\authDetails{Sarah Miller}{sarah@sarahholmesmiller.com}{0009-0007-8939-266X}[2]
\authDetails{Rebecca Milley}{rebeccamilley@mun.ca}{0000-0001-8091-8572}[1]

\affil[1]{ Memorial University of Newfoundland\\Canada}
\affil[2]{ University of Maryland\\College Park\\USA}

\date{}

\begin{document}
\maketitle

\begin{abstract}
    \noindent
    Milley and Renault proved an interesting characterisation of invertible
    elements in the dead-ending universe: they are the games with no
    subpositions of outcome $\mathscr{P}$ (the `$\mathscr{P}$-free' games). We
    generalise their approach to obtain a stronger result and show in
    particular that the set of $\mathscr{P}$-free blocking games is closed
    under addition, which yields that every $\mathscr{P}$-free blocking game is
    invertible modulo the blocking universe. This has consequences for the
    invertible subgroups of various other mis\`ere monoids.
\end{abstract}

\section{Introduction}

Invertibility in mis\`ere is complicated. In normal play, the situation is
simple: each game is invertible, and its inverse is its conjugate (the form
obtained by swapping the roles of Left and Right). So, the set of normal play
values yields a group. In the full mis\`ere universe $\set{M}$, however, there
are \emph{no} (non-zero) invertible games whatsoever, as shown by Mesdal and
Ottaway \cite[Theorem 7 on p.~5]{mesdal.ottaway:simplification}. Following the
breakthroughs of Plambeck and Siegel (see \cite{plambeck.siegel:misere,
plambeck:taming}), mis\`ere research has focused on \emph{restricted} sets of
games, where the comparison relation is relaxed, and invertible elements do
indeed appear.

A shocking discovery was that restricted mis\`ere monoids exist in which a game
can be invertible but its inverse is \emph{not} its conjugate. The first such
example was an impartial game found by Plambeck and Siegel \cite[A.6 on
p.~617]{plambeck.siegel:misere}. This was later followed by a partizan example
by Milley \cite[Corollary 4 on p.~11]{milley:partizan}, and Davies and Yadav
then found more \cite[pp.~24--25]{davies.yadav:invertibility}.

But the cruelty of mis\`ere relented somewhat for the work of Larsson,
Nowakowski, and Santos on Absolute Combinatorial Game Theory
\cite{larsson.nowakowski.ea:absolute}. They defined special sets of games
called \emph{(absolute) universes} \cite[Definition 23 on
p.~117]{larsson.nowakowski.ea:absolute}, which are closed under various
properties. In addition to the many things they were able to prove with these
objects, the usefulness of the universe was bolstered by the fact that
researchers had already been studying sets that happened to be universes: in
particular, the dicot universe $\set{D}$ and the dead-ending universe
$\set{E}$. A form is a \emph{dicot} if every proper subposition is a dicot and
either both players have an option or else no player has an option; it is
\emph{dead-ending} if, once a player runs out of moves, they can never move
again in any subsequent position. Indeed, it was then proven by Larsson,
Milley, Nowakowski, Renault, and Santos that both of these universes have the
behaviour that the inverse of a game (when it exists) is its conjugate
\cite[Theorems 25 and 27 on pp.~262--263]{larsson.milley.ea:reversibility};
such behaviour is often called the \emph{conjugate property} (see
\cite[Definition 2 on p.~248]{larsson.milley.ea:reversibility}).

Davies and Yadav then showed that \emph{every} universe has the conjugate
property \cite[Theorem 3.7 on p.~11]{davies.yadav:invertibility}, and also gave
a characterisation of the invertible elements in every universe \cite[Theorem
3.8 on p.~13]{davies.yadav:invertibility}. Characterisations had previously
been found for the dicot universe $\set{D}$ by Fisher, Nowakowski, and Santos
\cite[Theorem 12 on p.~7]{fisher.nowakowski.ea:invertible}, and also for the
dead-ending universe $\set{E}$ by Milley and Renault \cite[Theorems 19 and 22
on pp.~11--12]{milley.renault:invertible}.

As was noted in \cite[p.~14]{davies.yadav:invertibility}, the previous
characterisation of invertibility in $\set{E}$ is somewhat surprising when
compared with the more general characterisation due to Davies and Yadav, the
latter of which is almost recursive in requiring that proper subpositions be
invertible (among some other technical statements we do not reproduce here).
Milley and Renault were able to give a characterisation by referring only to
the outcomes of the subpositions of a game; in particular, by requiring that no
subposition be a $\mathscr{P}$-position (a win for the second player)---such a
form is called $\mathscr{P}$-free. It was posed as an open problem to
investigate when alternative characterisations like this exist for other
universes \cite[Open problem 3.11 on p.~14]{davies.yadav:invertibility}.

The arguments that Milley and Renault employed do not generalise to every
universe, but we will show they can be pushed significantly. They used a
concept called \emph{tipping points} to demonstrate that the set of
$\mathscr{P}$-free dead-ending games is closed under addition. In this paper,
we find larger sets of $\mathscr{P}$-free games that are closed under addition;
this will then have reasonably far-reaching consequences to many other mis\`ere
monoids and universes, and in particular to understanding their invertible
elements.

We follow roughly the same path taken by Milley and Renault in their proof,
although we achieve greater generality by considering a wider context, which
has the consequence of making the results appear somewhat technical. We attempt
a balance of abstraction with practicality here, being mindful of building
tools such that they may be wielded effectively in further research.

In \cref{sec:prelims}, we give some preliminaries and important definitions.
The reader familiar with mis\`ere should still find utility here as some ideas
are very recent.

In \cref{sec:tipping-points}, we build a generalisation of the tipping point
theory introduced in \cite{milley.renault:invertible}. In particular, we are
able to show that, if $\set{A}$ is a semigroup of games that is a subset of a
monoid satisfying various properties like being \emph{outcome-stable},
\emph{hereditary}, and \emph{integer-invertible} (to be defined later), then
the set of $\mathscr{P}$-free forms of $\set{A}$ (what we will call $\pf{A}$
later on) is closed under addition.

In \cref{sec:blocking}, we ground ourselves by showing how to apply the
technical results of \cref{sec:tipping-points} to the recently discovered
blocking universe $\set{B}$, and what implications this has for when the set of
$\mathscr{P}$-free forms of a monoid is a subgroup of the invertible subgroup
of the monoid. A game is \emph{blocking} (or \emph{blocked}) if, whenever Left
cannot move and Right opens up a move for her, then Left can move back to a
position from which they cannot move (and symmetrically for Right). Every
dead-ending form is trivially blocking, and it is easy to see that $\set{B}$ is
significantly larger than $\set{E}$, which was discussed in
\cite{davies.mckay.ea:pocancellation}.

Finally, in \cref{sec:final-remarks}, we suggest some irresistible directions
for further research---it is high time that we start to understand the group
structure of the invertible subgroups of various mis\`ere monoids and
universes.

\section{Preliminaries}
\label{sec:prelims}

We will briefly recall some important definitions here, but the reader is
directed to Siegel \cite{siegel:combinatorial} for a more thorough treatment
and more extensive background.

Recall that, given a set of games $\set{A}$, we define the restricted relation
`$\geq_\set{A}$' on all pairs of games $G,H\in\set{M}$ as follows:
\[
    G\geq_\set{A}H\iff\outcome(G+X)\geq\outcome(H+X)\quad\forall X\in\set{A}.
\]
Sometimes other authors do restrict $G$ and $H$ to being elements of $\set{A}$
themselves, but we do not do that here. Nor do we make any assumptions about
what properties $\set{A}$ might satisfy (like being additively closed).

Recall that a game $G$ is called \emph{Left $\set{U}$-strong} if
$\outcomeL(G+X)=\mathscr{L}$ for every Left end $X$ in $\set{U}$ (see
\cite[Definition 2.1 on p.~4]{davies.yadav:invertibility}), and similarly for
being \emph{Right $\set{U}$-strong}. This concept is critical for the
comparison of games modulo universes, which are sets of games with additive,
hereditary, conjugate, and option closure (the details of which are not
important to us here).

\begin{theorem}[{\cite[Theorem 4 on p.~103]{larsson.nowakowski.ea:absolute}}]
    \label{thm:comparison}
    If $\set{U}$ is a universe and $G,H\in\set{M}$, then $G\geq_\set{U}H$ if
    and only if $G$ and $H$ satisfy:
    \begin{enumerate}
        \item
            for every $G^R$, either there exists some $H^R$ with
            $G^R\geq_\set{U} H^R$, or else there exists some $G^{RL}$ with
            $G^{RL}\geq_\set{U}H$; and
        \item
            for every $H^L$, either there exists some $G^L$ with
            $G^L\geq_\set{U}H^L$, or else there exists some $H^{LR}$ with
            $G\geq_\set{U}H^{LR}$.
        \item
            if $H$ is a Left end, then $G$ is Left $\set{U}$-strong; and
        \item
            if $G$ is a Right end, then $H$ is Right $\set{U}$-strong.
    \end{enumerate}
\end{theorem}
The first two conditions in \cref{thm:comparison} are commonly referred to as
the \emph{maintenance property}, and the latter two the \emph{proviso}.

When we refer to a \emph{subposition} of a game $G$, we follow Siegel
\cite[Definition 1.3 on pp.9--10]{siegel:combinatorial} in meaning a game that
can be reached by any (possibly empty and not necessarily alternating) sequence
of moves on $G$, and reserve \emph{proper subposition} to mean those
subpositions of $G$ distinct from $G$ itself.

Recall also that the \emph{rank} of a game form is the height of its game tree
(i.e.\ the length of a longest run)---some authors might prefer \emph{formal
rank}, but we do not use that here, instead using rank to agree with
\emph{formal birthday} as used in normal play: $\birth(G)$.

Given the importance of Siegel's simplest forms (see \cite[\S5]{siegel:on}) in
the characterisation of mis\`ere invertibility (see
\cite[\S3]{davies.yadav:invertibility}), it seems wise to consider our results
in the wider context of augmented forms $\maug\supsetneq\set{M}$. These forms
introduce new equivalence classes, and so such a strengthening is not always
trivial, but we make an effort to do so in this work. The reader is encouraged
to read Siegel's original treatment; in particular, \cite[Definition 5.1 on
p.~212]{siegel:on} gives the definition of augmented forms. The motivations and
intricacies will not be fundamental to our arguments, so we do not go into
detail here.

The main thing that the reader need keep in mind is that an augmented form is
obtained from a form in $\set{M}$ by possibly adding an abstract symbol
`$\tomb$', called a \emph{tombstone}, to the Left or Right options (or both) of
any number of subpositions of the game. For example, $\{\tomb,0\mid*\}$ is an
augmented form. If a Left (Right) tombstone is present, then Left (Right) is
defined to win going first. So, even though Left loses $\{0\mid*\}$ going
first, she wins $\{\tomb,0\mid*\}$ going first on account of there being a Left
tombstone.

We also recall that a game is called \emph{Left end-like} if either it is a
Left end (i.e.\ Left has no options) or else it has a Left tombstone; and being
\emph{Right end-like} is defined analogously (see \cite[Definition 5.2 on
p.~212]{siegel:on}). For example, $\{\tomb,0\mid*\}$ is Left end-like but not
Right end-like, whereas $\{\cdot\mid1,\tomb\}$ is both Left and Right end-like.
It is an obvious remark (given \cite[Definition 5.4 on p.~213]{siegel:on}) that
a sum of games $G+H$ is Left (Right) end-like if and only if $G$ and $H$ are
both Left (Right) end-like, which we will use later.

From this point onwards, we will generally use the term `game' rather than
`augmented game'. It will be clear in context when we are using arguments
specific to tombstones.

Recall that a set of games $\set{A}$ is called \emph{hereditary} if, for every
$G\in\set{A}$, we have $G'\in\set{A}$ for all options $G'$ of $G$. Hereditary
semigroups\footnote{
    Although some authors allow a semigroup to be empty, we do not do so here.
}
of games must in fact be monoids, because they must contain the strict form of
0. As such, in what follows, we will write `hereditary monoid' in place of
`hereditary semigroup'.

Milley and Renault called a dead-ending game \emph{$\mathscr{P}$-free} if, in
reduced form, it had no subposition of outcome $\mathscr{P}$ \cite[Definition 1
on p.~2]{milley.renault:invertible}. We split this idea into two definitions:
one defined on the strict form; and the other on the equivalence class. Each
will have its own use, and we will maintain compatibility with the previous
definition in this way.

\begin{definition}[cf.\ {\cite[Definition 1 on
    p.~2]{milley.renault:invertible}}]
    We say a game is \emph{(strictly) $\mathscr{P}$-free} if no subposition has
    outcome $\mathscr{P}$. If $\set{A}$ is a set of games, then we write
    $\pf{A}$ for the set of (strictly) $\mathscr{P}$-free elements of
    $\set{A}$.
\end{definition}

When discussing strictly $\mathscr{P}$-free games, we will usually drop
`strictly' except in cases where we want to emphasise the distinction between
this and \cref{def:p-free-equiv} below. Indeed, it is important to clarify that
we are defining the property of being strictly $\mathscr{P}$-free on strict
game forms, rather than on equivalence classes (the latter of which was done in
\cite{milley.renault:invertible}). It is easy to find two games, one of which
is $\mathscr{P}$-free, and one of which is not, that are equal to each other:
for example, consider $\{\{0,\overline{1}\mid0\}\mid\cdot\}$ and
$\{\{0,\overline{1}\mid0\},*\mid\cdot\}$, which are equal modulo \emph{every}
set of games (choose any you like), yet the former is strictly
$\mathscr{P}$-free, whereas the latter is not. This motivates the following
definition, which captures Milley and Renault's original notion, being
well-defined up to equivalence (modulo $\set{A}$) rather than isomorphism.

\begin{definition}[cf.\ {\cite[Definition 1 on
    p.~2]{milley.renault:invertible}}]
    \label{def:p-free-equiv}
    If $\set{A}$ is a set of games, then we say a game $G$ is
    \emph{$\mathscr{P}$-free modulo $\set{A}$} if there exists a strictly
    $\mathscr{P}$-free form $H\in\set{A}$ such that $G\equiv_\set{A}H$. We
    write $\pf{A}[S]$ for the elements of $\set{A}$ that are $\mathscr{P}$-free
    modulo some set $\set{S}$.
\end{definition}

In this paper, we will only have cause to consider $\pf{A}[A]$ (i.e.\ with
$\set{A}=\set{S}$ in \cref{def:p-free-equiv}), but one could easily imagine
that the more general definition would be useful on occasion.

Note that an augmented form with a tombstone will never have outcome
$\mathscr{P}$, but of course such forms are not always $\mathscr{P}$-free:
take, for example, our game $\{\tomb,0\mid*\}$ from earlier, since it has $*$
as a subposition (which has outcome $\mathscr{P}$). But a form like
$\{\tomb,0\mid\cdot\}$ is indeed $\mathscr{P}$-free.

Milley and Renault used the machinery similar to that of the following section
to show that the set of $\mathscr{P}$-free dead-ending games is closed under
addition \cite[Theorem 16 on p.~10]{milley.renault:invertible}; albeit with
different terminology, they showed that both $\pf{E}$ and $\pf{E}[E]$ are
monoids (it is trivial that 0 is an identity element in both semigroups). An
optimist would ask whether it were possible that the set of \emph{all}
$\mathscr{P}$-free games (of $\set{M}$) could be closed under addition. In
typical mis\`ere fashion, the answer is, of course, no. To see this, consider
the sum
\[
    \{\cdot\mid\{1\mid1\}\}+\{\overline{1}\mid\overline{1}\}.
\]
One can easily verify that $\{\cdot\mid\{1\mid1\}\}$ and
$\{\overline{1}\mid\overline{1}\}$ are both $\mathscr{P}$-free forms, but that
their sum has outcome $\mathscr{P}$---and hence is certainly not
$\mathscr{P}$-free! It is straightforward to show that there exists no
counter-example $G+H$ of formal birthday (strictly) less than 5, and so the one
we have presented here is minimal.

Note it is also easy to find counter-examples of the form $G+\overline{G}$
(which are symmetric forms; i.e.\ they are isomorphic to their conjugates),
where $G$ (and hence also $\overline{G}$) is $\mathscr{P}$-free. Take, for
example, $G=\{\cdot\mid\{1\mid0,1\}\}$.

Those familiar with mis\`ere may recall the striking theorem of Mesdal and
Ottaway \cite[Theorem 3 on p.4]{mesdal.ottaway:simplification} that says if
$o_1$, $o_2$, and $o_3$ are outcomes (i.e.\ any of $\mathscr{L}$,
$\mathscr{N}$, $\mathscr{P}$, and $\mathscr{R}$), then there exist games
$G,H\in\set{M}$ such that
\begin{align*}
	\outcome(G)&=o_1,\\
	\outcome(H)&=o_2,\text{ and}\\
	\outcome(G+H)&=o_3.
\end{align*}
It would be tempting to say that this result should have already prevented us
from even asking our question about whether the set of all $\mathscr{P}$-free
forms is closed under addition, by setting, for example,
$o_1,o_2\neq\mathscr{P}$ and $o_3=\mathscr{P}$. But it is not so simple, since
there is no guarantee made by the theorem that $G$ and $H$ can be chosen to be
$\mathscr{P}$-free. Indeed, in all of the relevant examples $G+H$ that Mesdal
and Ottaway give in \cite[Appendix on
pp.~11--12]{mesdal.ottaway:simplification}, none use $G$ and $H$ both
$\mathscr{P}$-free.

We will show (in \cref{thm:tipping-points}) that for $\mathscr{P}$-free games
of a particular quality, the sequence of outcomes of $G+n$ (where $n$ varies
over all integers) splits up into exactly three contiguous components: an
infinite sequence of $\mathscr{L}$, followed by a finite sequence of
$\mathscr{N}$, followed by an infinite sequence of $\mathscr{R}$. (The
boundaries of these components---the integers where the outcome \emph{tips
over} from one to another---are the tipping points that we referred to
earlier.) We will use these particular $\mathscr{P}$-free games to start
pointing towards the structure of the invertible subgroups of various mis\`ere
monoids.

\section{Tipping point theory}
\label{sec:tipping-points}

Milley and Renault's ideas of tipping points come from the observation that,
given \emph{any} game $G$, if we add a positive integer $n$ of sufficiently
large rank, then the outcome of $G+n$ must be $\mathscr{R}$. So, there must
exist a non-negative integer $\rtp(G)$ of minimal rank such that the outcome of
$G+n$ is $\mathscr{R}$ for all $n\geq\rtp(G)$. This $\rtp(G)$ is called the
\emph{$\mathscr{R}$-tipping point of $G$}, and it is clear by symmetry that we
can define the $\mathscr{L}$-tipping point in a similar way. It turns out that
it also makes sense to define an $\mathscr{N}$-tipping point, although a
$\mathscr{P}$-tipping point would \emph{not} be well-defined in general. We
will now give formal definitions and show that they are well-defined.

As we mentioned earlier, our treatment here roughly follows the shape of the
treatment by Milley and Renault, and we will give reference citations to the
similar results at each appropriate step.

\subsection{Tipping point basics}
\label{sec:tp-basics}

\begin{definition}[cf.\ {\cite[Definition 5 on
    p.~6]{milley.renault:invertible}}]
    \label{def:tipping-points}
    If $G$ is a game, then we say the \emph{$\mathscr{N}$-tipping point} of $G$
    is the smallest non-negative integer $\ntp(G)$ such that
    \begin{align*}
        \outcome(G+\ntp(G))&=\mathscr{N},\text{ or}\\
        \outcome(G+\overline{\ntp(G)})&=\mathscr{N};
    \end{align*}
    the \emph{$\mathscr{R}$-tipping point} of $G$ is the smallest non-negative
    integer $\rtp(G)$ such that
    \[
        \outcome(G+\rtp(G))=\mathscr{R};
    \]
    and the \emph{$\mathscr{L}$-tipping point} of $G$ is the smallest
    non-negative integer $\ltp(G)$ such that
    \[
        \outcome(G+\overline{\ltp(G)})=\mathscr{L}.
    \]
\end{definition}

As an example for why the concept of a $\mathscr{P}$-tipping point is not
always well-defined, consider that no sum of integers has outcome
$\mathscr{P}$, and hence no integer could have a $\mathscr{P}$-tipping point.
Investigating variations of $\mathscr{P}$-tipping points (by adding things
other than integers), such as the Right-$r$-commission in
\cite{davies.mckay.ea:pocancellation}, might be worth pursuing further. It may
also be worthwhile to investigate when the $\mathscr{P}$-tipping point
\emph{would} exist, and how it behaves. But such things will not be of use to
us here.

Although the $\mathscr{L}$-, $\mathscr{N}$-, and $\mathscr{R}$-tipping points
have been defined on strict forms, it is a simple observation that, if
$\set{A}$ is a semigroup of games containing $1$ and $\overline{1}$, then the
tipping points are well-defined up to equivalence modulo $\set{A}$. (In fact,
$\set{A}$ need not be additively closed, as containing every integer would be
enough.)

Note that we will use negative signs \emph{exclusively} for when we are
interpreting integers as elements of $\mathbb{Z}$, and overlines for when we
wish to refer to the conjugate of an integer interpreted as a game. For
example, if we have a game $G$, then $G+(3-2)$ should be interpreted as the
game plus the integer $1$, where $3-2$ should be treated as an expression in
$\mathbb{Z}$. It would be \emph{incorrect} to interpret this as
$G+3+\overline{2}$! Similarly, (game) integers are generally incomparable in
mis\`ere (e.g.\ we do \emph{not} have $1\leq2$), and as such it will generally
be clear that, when we are comparing integers, we are treating them as elements
of $\mathbb{Z}$. We will be passing between $\mathbb{Z}$ and (game) integers
freely for the duration of this study and without any further mention.

We now show that the $\mathscr{L}$-, $\mathscr{N}$-, and $\mathscr{R}$-tipping
points do indeed always exist.

\begin{theorem}[cf.\ {\cite[pp.~5--6]{milley.renault:invertible}}]
    \label{tipping-point-existence}
    If $G$ is a game, then $\ltp(G)$, $\rtp(G)$, and $\ntp(G)$ all exist.
\end{theorem}

\begin{proof}
    Observe that $\outcome(G+\birth(G)+1)=\mathscr{R}$, and
    $\outcome(G+\overline{\birth(G)+1})=\mathscr{L}$ by symmetry. Thus, by the
    well-ordering principle, it is clear that the $\mathscr{L}$- and
    $\mathscr{R}$-tipping points always exist.

    We now consider the $\mathscr{N}$-tipping point. If
    $\outcome(G)=\mathscr{N}$, then we are done: $\ntp(G)=0$.

    Assume now that $\outcome(G)=\mathscr{L}$. We know that
    $\outcome(G+\rtp(G))=\mathscr{R}$ (and $\rtp(G)\geq1$). Now, if
    $\outcomeR(G+(\rtp(G)-1))=\mathscr{L}$, then
    $\outcomeL(G+\rtp(G))=\mathscr{L}$, which is not the case, and hence
    $\outcome(G+(\rtp(G)-1))=\mathscr{N}$. The result is then immediate from
    the well-ordering principle.

    If $\outcome(G)=\mathscr{R}$, then it is clear by symmetry that
    $\ntp(G)=\ntp(\overline{G})$, which we have already shown is well-defined.

    Finally, suppose $\outcome(G)=\mathscr{P}$; then
    $\outcome(G+1)\geq\mathscr{N}$ and
    $\outcome(G+\overline{1})\leq\mathscr{N}$. If either of these are an
    equality, then we are done: $\ntp(G)$ would simply be 1. Otherwise,
    $\outcome(G+1)=\mathscr{L}$, and there must exist some (positive) integer
    $k$ such that $\outcome(G+1+k)=\mathscr{N}$ by previous arguments, and
    hence $\ntp(G)$ exists by the well-ordering principle.
\end{proof}

If we were to consider monoids containing integers in which
$1+\overline{1}\equiv0$ (modulo the monoid)\footnote{We will later refer to
this as an \emph{integer-invertible} monoid (cf.\ \cref{def:integral}).}, and
$G$ were a game that did not have outcome $\mathscr{N}$, then
$\ntp(G)=1+\min(\ntp(G+1),\ntp(G+\overline{1}))$. In general, however, this
formula does not hold: take, for example, the game
$G=\{\cdot\mid\{\cdot\mid\{\{1\mid\overline{1}\}\mid\cdot\}\}\}$, which
satisfies $\outcome(G)=\outcome(G+1)=\outcome(G+\overline{1})=\mathscr{L}$, and
$\outcome(G+1+\overline{1})=\mathscr{N}$. The formula fails as
$\ntp(G+\overline{1})=1$, and so we would predict $\ntp(G)=2$, while in fact
$\ntp(G)=3$. Note $G$ is not $\mathscr{P}$-free.

\begin{lemma}[cf.\ {\cite[Lemma 6 on p.~6]{milley.renault:invertible}}]
    \label{lem:p-free-plus-integer}
    If $G$ is $\mathscr{P}$-free and $n$ is an integer, then $G+n$ is also
    $\mathscr{P}$-free.
\end{lemma}

\begin{proof}
    Without loss of generality, and by induction, it suffices to prove that
    $G+1$ is $\mathscr{P}$-free. We first show that $G+1$ cannot have outcome
    $\mathscr{P}$.

    If Left wins $G+1$ playing second, then at some point Left makes a winning
    move from $G'+1$ to $G'$, for a subposition $G'$ (of $G$) obtained by
    alternating play in $G$, where $G'$ has outcome $\mathscr{L}$ or
    $\mathscr{P}$. Since it is a subposition of $G$, and $G$ is
    $\mathscr{P}$-free, we must have $\outcome(G')=\mathscr{L}$. But then the
    same strategy played on $G$ alone means Left wins $G$ playing second, and
    so Left wins $G+1$ playing first (by playing to $G$). Hence,
    $\outcome(G+1)\neq\mathscr{P}$.

    Now, all subpositions of $G+1$ are of the form $G'+1$ and $G'$, where $G'$
    is a subposition of $G$ that is necessarily $\mathscr{P}$-free. It is then
    clear by induction that each such proper subposition of $G+1$ is
    $\mathscr{P}$-free (note that the game $1$ is $\mathscr{P}$-free), and thus
    we conclude that $G+1$ is $\mathscr{P}$-free.
\end{proof}

We mentioned earlier that no (sums of) integers have outcome $\mathscr{P}$ and
that this would imply immediately that every integer is $\mathscr{P}$-free.
\cref{lem:p-free-plus-integer} also gives us a one-line proof of this fact upon
observing that 0 is $\mathscr{P}$-free.

\begin{definition}[cf.\ {\cite[Theorem 4 on p.~4]{milley.renault:invertible}}]
    \label{def:outcome-stable}
    We say a set of games $\set{A}$ is \emph{outcome-stable} if the following
    are true for all $G,H\in\pf{A}$:
    \begin{enumerate}
        \item
            if $\outcome(G),\outcome(H)=\mathscr{L}$, then
            $\outcome(G+H)=\mathscr{L}$;
        \item
            if $\outcome(G),\outcome(H)=\mathscr{R}$, then
            $\outcome(G+H)=\mathscr{R}$;
        \item
            if $\outcome(G)=\mathscr{L}$ and $\outcome(H)=\mathscr{N}$, then
            $\outcomeL(G+H)=\mathscr{L}$; and
        \item
            if $\outcome(G)=\mathscr{R}$ and $\outcome(H)=\mathscr{N}$, then
            $\outcomeR(G+H)=\mathscr{R}$.
    \end{enumerate}
\end{definition}

Milley and Renault showed that the dead-ending universe $\set{E}$ is
outcome-stable \cite[Theorem 4 on p.~4]{milley.renault:invertible}, which is
how this definition came to be. We will show in \cref{sec:blocking}
(specifically \cref{lem:B-outcome-stable}) that the universe of blocking games
$\set{B}$ is also outcome-stable (and recall that this is a significantly
larger universe than $\set{E}$), which would in fact imply the result of Milley
and Renault since a subset of an outcome-stable set must itself be
outcome-stable. Also observe that, for a conjugate-closed set of games,
conditions (1) and (2) of \cref{def:outcome-stable} are equivalent, as well as
(3) and (4).

For $\mathscr{P}$-free games in an outcome-stable set, note that the
relationship between the outcomes of the summands ($G$ and $H$) and the outcome
of the sum ($G+H$) is precisely the same relationship that one enjoys when
working in normal play: adding two games of outcome $\mathscr{L}$ results in a
game of outcome $\mathscr{L}$, et cetera.

Recall from earlier our sum of $\mathscr{P}$-free games
$\{\cdot\mid\{1\mid1\}\}+\{\overline{1}\mid\overline{1}\}$ that had outcome
$\mathscr{P}$. Notice that $\outcome(\{\cdot\mid\{1\mid1\}\})=\mathscr{N}$ and
$\outcome(\{\overline{1}\mid\overline{1}\})=\mathscr{L}$. As such, if $\set{A}$
is a set of games containing both of these forms, then $\set{A}$ is \emph{not}
outcome-stable. In particular, we observe that $\set{M}$ (and hence also
$\maug$) is not outcome-stable. It would therefore be interesting to find sets
of games that are maximal with respect to being outcome-stable, which we leave
as an open problem.

\begin{problem}
    What is an example of a set of $\mathscr{P}$-free games that is maximal
    with respect to being outcome-stable?
\end{problem}

The $\mathscr{P}$-free forms of outcome-stable semigroups of games $\set{A}$
that contain 1 and $\overline{1}$ exhibit a wonderful structure in their
tipping-points: the three contiguous components of outcomes that we spoke of
earlier. See \cref{fig:contiguous} for a visualisation.

\begin{figure}
\centering

\def\a{-7}
\def\b{1.2}
\def\c{3.5}
\def\d{7}
\begin{tikzpicture}[scale=0.8]
\node at (\a+1,1.5) {\large{$\textrm{o}(G)=\mathscr{L}$:}};
\node at (0,0.6) {$0$};
\node at (\b,0.6) {$\ntp(G)$};
\node at (\c,0.6) {$\rtp(G)$};
\node at (\d+0.3,0.05) {$k$};
\draw[very thick,blue,<-] (\a,0)--(\b,0);
\draw[very thick,green] (\b+.02,0)--(\c,0);
\draw[very thick,red,->] (\c+.02,0)--(\d,0);
\node at (0,0) {$|$};
\node at (\b-0.03,0) {\bf{)}};
\node at (\b+.03,0) {\bf{[}};
\node at (\c-.02,0) {\bf{)}};
\node at (\c+.02,0) {\bf{[}};
\draw [thick,decorate,decoration={brace,amplitude=10pt,mirror,raise=-0.5cm}]
    (\a,-1)--(\b-.02,-1)
    node[midway,yshift=-5pt]{$\scalebox{0.75}{$\textrm{o}(G+k)=\mathscr{L}$}$};
\draw [thick,decoration={brace,amplitude=10pt,mirror,raise=-0.5cm},
    decorate] (\b+.02,-1)--(\c-.02,-1)
    node[midway,yshift=-5pt]{$\scalebox{0.75}{$\textrm{o}(G+k)=\mathscr{N}$}$};
\draw [thick,decoration={brace,amplitude=10pt,mirror,raise=-0.5cm},
    decorate] (\c+.02,-1)--(\d,-1)
    node[midway,yshift=-5pt]{$\scalebox{0.75}{$\textrm{o}(G+k)=\mathscr{R}$}$};
\end{tikzpicture}

\vspace{0.5cm}

\def\a{-7}
\def\b{-1.5}
\def\c{2}
\def\d{7}
\begin{tikzpicture}[scale=0.8]
\node at (\a+1,1.5) {\large{$\textrm{o}(G)=\mathscr{N}$:}};
\node at (0,0.6) {$0$};
\node at (\b,0.6) {$\ltp(G)$};
\node at (\c,0.6) {$\rtp(G)$};
\node at (\d+0.3,0.05) {$k$};
\draw[very thick,blue,<-] (\a,0)--(\b,0);
\draw[very thick,green] (\b+.02,0)--(\c,0);
\draw[very thick,red,->] (\c+.02,0)--(\d,0);
\node at (0,0) {$|$};
\node  at (\b-0.03,0) {\bf{]}};
\node at (\b+.03,0) {\bf{(}};
\node at (\c-.02,0) {\bf{)}};
\node at (\c+.02,0) {\bf{[}};
\draw [thick,decoration={brace,amplitude=10pt,mirror,raise=-0.5cm},
    decorate] (\a,-1)--(\b-.02,-1)
    node[midway,yshift=-5pt]{$\scalebox{0.75}{$\textrm{o}(G+k)=\mathscr{L}$}$};
\draw [thick,decoration={brace,amplitude=10pt,mirror,raise=-0.5cm},
    decorate] (\b+.02,-1)--(\c-.02,-1)
    node[midway,yshift=-5pt]{$\scalebox{0.75}{$\textrm{o}(G+k)=\mathscr{N}$}$};
\draw [thick,decoration={brace,amplitude=10pt,mirror,raise=-0.5cm},
    decorate] (\c+.02,-1)--(\d,-1)
    node[midway,yshift=-5pt]{$\scalebox{0.75}{$\textrm{o}(G+k)=\mathscr{R}$}$};
\end{tikzpicture}

\vspace{0.5cm}

\def\a{-7}
\def\b{-3.5}
\def\c{-1.2}
\def\d{7}
\begin{tikzpicture}[scale=0.8]
\node at (\a+1,1.5) {\large{$\textrm{o}(G)=\mathscr{R}$:}};
\node at (0,0.6) {$0$};
\node at (\b,0.6) {$\ltp(G)$};
\node at (\c,0.6) {$\ntp(G)$};
\node at (\d+0.3,0.05) {$k$};
\draw[very thick,blue,<-] (\a,0)--(\b,0);
\draw[very thick,green] (\b+.02,0)--(\c,0);
\draw[very thick,red,->] (\c,0)--(\d,0);
\node  at (0,0) {$|$};
\node at (\b-0.03,0) {\bf{]}};
\node at (\b+.03,0) {\bf{(}};
\node at (\c-.02,0) {\bf{]}};
\node at (\c+.02,0) {\bf{(}};
\draw [thick,decoration={brace,amplitude=10pt,mirror,raise=-0.5cm},
    decorate] (\a,-1)--(\b-.02,-1)
    node[midway,yshift=-5pt]{$\scalebox{0.75}{$\textrm{o}(G+k)=\mathscr{L}$}$};
\draw [thick,decoration={brace,amplitude=10pt,mirror,raise=-0.5cm},
    decorate] (\b+.02,-1)--(\c-.02,-1)
    node[midway,yshift=-5pt]{$\scalebox{0.75}{$\textrm{o}(G+k)=\mathscr{N}$}$};
\draw [thick,decoration={brace,amplitude=10pt,mirror,raise=-0.5cm},
    decorate] (\c+.02,-1)--(\d,-1)
    node[midway,yshift=-5pt]{$\scalebox{0.75}{$\textrm{o}(G+k)=\mathscr{R}$}$};
\end{tikzpicture}

\caption{A visualisation of \cref{thm:tipping-points}, showing the contiguous
components of the sequence of outcomes $\outcome(G+k)$.}
\label{fig:contiguous}

\end{figure}

\begin{theorem}[cf.\ {\cite[pp.~5--6]{milley.renault:invertible}}]
    \label{thm:tipping-points}
    If $\set{A}$ is an outcome-stable semigroup containing $1$ and
    $\overline{1}$, and $G\in\pf{A}[A]$, then the sequence of outcomes
    $\outcome(G+k)$ contains three contiguous components according to the
    following rules:
    \[
        \outcome(G+k)=\begin{cases}
            \mathscr{L} & k\leq -\ltp(G),\\
            \mathscr{L} & 0\leq k<\ntp(G)\text{ and }\outcome(G)=\mathscr{L},\\
            \mathscr{N} & -\ltp(G)<k\leq-\ntp(G)\text{ and
            }\outcome(G)=\mathscr{R},\\
            \mathscr{N} & -\ltp(G)<k<\rtp(G)\text{ and
            }\outcome(G)=\mathscr{N},\\
            \mathscr{N} & \ntp(G)\leq k<\rtp(G)\text{ and
            }\outcome(G)=\mathscr{L},\\
            \mathscr{R} & 0\geq k>-\ntp(G)\text{ and
            }\outcome(G)=\mathscr{R},\\
            \mathscr{R} & k\geq \rtp(G),
        \end{cases}
    \]
    for all integers $k$.
\end{theorem}

\begin{proof}
    We may assume that $G\in\pf{A}$. We know that
    $\outcome(G+\overline{\ltp(G)})=\mathscr{L}$. Now let $k$ be an integer
    satisfying $k>\ltp(G)$. Since $\set{A}$ is an outcome-stable semigroup, it
    follows from \cref{tipping-point-existence,lem:p-free-plus-integer} that
    \[
        \outcome(G+\overline{k})=
        \outcome(\underbrace{G+\overline{\ltp(G)}}_{\mathscr{L}}+
        \underbrace{\overline{k-\ltp(G)}}_{\mathscr{L}})=\mathscr{L}.
    \]
    By symmetry, it follows that $\outcome(G+k)=\mathscr{R}$ for all $k\geq
    \rtp(G)$.

    If $\outcome(G)=\mathscr{N}$, then the result follows immediately from
    \cref{tipping-point-existence}.

    Otherwise, by symmetry, it suffices to consider when
    $\outcome(G)=\mathscr{L}$. Recall from \cref{def:tipping-points} that
    $\outcome(G+\ntp(G))=\mathscr{N}$, and also that
    $\outcome(G+0)=\outcome(G)=\mathscr{L}$. Now, if $0<k<\ntp(G)$, then we
    know by \cref{def:tipping-points} that $\outcome(G+k)\neq\mathscr{N}$. By
    \cref{lem:p-free-plus-integer}, we have that
    $\outcome(G+k)\neq\mathscr{P}$. Since Left has an option to $G+(k-1)$,
    which has outcome $\mathscr{L}$ by induction, it follows that $G+k$ must
    itself have outcome $\mathscr{L}$.

    The only remaining case is where $\ntp(G)<k<\rtp(G)$ (and still
    $\outcome(G)=\mathscr{L}$). Again, since $\set{A}$ is an outcome-stable
    semigroup, it follows from
    \cref{tipping-point-existence,lem:p-free-plus-integer} that
    \[
        \outcome(G+k)=
        \outcome(\underbrace{G+\ntp(G)}_{\mathscr{N}}+
        \underbrace{(k-\ntp(G))}_{\mathscr{R}})=\mathscr{N}.
    \]
\end{proof}

Note that we do \emph{not} require 1 and $\overline{1}$ to be invertible here
(as inverses of each other). This is unlike the argument in
\cite[pp.~5--6]{milley.renault:invertible}, as integers are naturally
invertible in $\set{E}$.

Even though we have shown the $\mathscr{L}$-, $\mathscr{R}$-, and
$\mathscr{N}$-tipping points are well-defined for games of all outcomes,
\cref{thm:tipping-points} really does require $G$ to be $\mathscr{P}$-free in
the hypothesis. For example, take the dead-ending universe $\set{E}$, which we
know is an outcome-stable semigroup of games containing 1 and $\overline{1}$.
Now consider the game $*+*\in\set{E}$, and observe that
\begin{align*}
    \outcome(*+*)&=\mathscr{N},\\
    \outcome(*+*+1)&=\mathscr{P},\\
    \outcome(*+*+2)&=\mathscr{N},\rlap{ and}\\
    \outcome(*+*+n)&=\mathscr{R}\rlap{\quad for all $n\geq3$,}
\end{align*}
seemingly violating the behaviour asserted by \cref{thm:tipping-points}. But
$*+*$ is not $\mathscr{P}$-free, since the subposition $*$ has outcome
$\mathscr{P}$, and so \cref{thm:tipping-points} makes no assertion at all here.
It is perhaps interesting to observe, however, that any $\mathscr{P}$ appearing
in the sequence $(\outcome(G+n))_{n\in\mathbb{Z}}$ must be isolated (whether or
not $G$ is $\mathscr{P}$-free); i.e.\ there can exist no adjacent copies of
$\mathscr{P}$. To see this, recall that a game $G$ with outcome $\mathscr{P}$
must satisfy $\outcome(G+1)\geq\mathscr{N}\geq\outcome(G+\overline{1})$; then,
since $\mathscr{P}$ is incomparable with $\mathscr{N}$ in the partial order of
outcomes, neither $G+1$ nor $G+\overline{1}$ may have outcome $\mathscr{P}$.

We adopt a similar structure of proof in the following technical lemmas to that
of Milley and Renault \cite{milley.renault:invertible} in order to establish
the relationships between the tipping points of a form and those of its
options. These will help us later prove how the tipping points inform the
outcomes of sums of games. 

Although symmetric results are often left unstated, we declare their existence
for reference during a particularly technical proof that will follow
(\cref{thm:final-piece}). In case of any doubt left in the reader's mind as to
what the symmetric statements are, we include them in summary tables where
everything is explicit (\cref{tab:techs,tab:11-14,tab:combo} in
\cref{app:tables}).

\begin{lemma}[cf. {\cite[Lemmas 7--9 on pp.~6--7]{milley.renault:invertible}}]
    \label{lem:tech1}
    If $\set{A}$ is an outcome-stable, hereditary monoid containing $1$ and
    $\overline{1}$, and $G\in\pf{A}$, then $\ntp(G^L)\leq\rtp(G)$ for all Left
    options $G^L$ of $G$ with $\outcome(G^L)\neq\mathscr{R}$. (The symmetric
    result holds for $\ntp(G^R)\leq\ltp(G)$ when
    $\outcome(G^R)\neq\mathscr{L}$.)
\end{lemma}

\begin{proof}
    Since $\outcome(G+\rtp(G))=\mathscr{R}$, it follows that
    $\outcomeR(G^L+\rtp(G))=\mathscr{R}$ for all Left options $G^L$ of $G$; and
    hence $\outcome(G^L+\rtp(G))\leq\mathscr{N}$. If
    $\outcome(G^L)\neq\mathscr{R}$ then the result $\ntp(G^L)\leq\rtp(G)$
    follows from \cref{thm:tipping-points}.
\end{proof}

\begin{lemma}
    \label{lem:left-end-n-r1}
    If $G$ is a $\mathscr{P}$-free Left end with $\outcome(G)=\mathscr{N}$,
    then $\rtp(G)=1$.
\end{lemma}

\begin{proof}
    Playing first on $G+1$, Left must move to $G$ (since $G$ is a Left end),
    leaving Right to move on a position with outcome $\mathscr{N}$. Thus, Right
    wins going second, which means $\outcome(G+1)\leq\mathscr{P}$. By
    \cref{lem:p-free-plus-integer}, we know $G+1$ is $\mathscr{P}$-free, and
    hence $\outcome(G+1)=\mathscr{R}$, yielding the result.
\end{proof}

\begin{lemma}[cf. {\cite[Lemma 7 on p.~6]{milley.renault:invertible}}]
    \label{lem:tech2}
    If $\set{A}$ is an outcome-stable, hereditary monoid containing $1$ and
    $\overline{1}$, and $G\in\pf{A}$ with $\outcome(G)=\mathscr{N}$, then
    either $G$ is Left end-like with $\rtp(G)=1$, or else there exists a Left
    option $G^L$ of $G$ with $\outcome(G^L)=\mathscr{L}$ such that
    $\ntp(G^L)=\rtp(G)$. (The symmetric result holds for $\ntp(G^R)=\ltp(G)$.)
\end{lemma}

\begin{proof}
    By \cref{thm:tipping-points}, recall that
    $\outcome(G+(\rtp(G)-1))=\mathscr{N}$. As such, either $G+(\rtp(G)-1)$ is
    Left end-like or else Left has a winning move.

    If $G+(\rtp(G)-1)$ is a Left end, then $\rtp(G)=1$ by
    \cref{lem:left-end-n-r1}. If $G+(\rtp(G)-1)$ has a Left tombstone, then
    both $G$ and $\rtp(G)-1$ must be Left end-like; but $\rtp(G)-1$ is a
    non-negative integer, and hence it is Left end-like if and only if
    $\rtp(G)=1$. So, assume now that Left has a winning move.

    By \cref{lem:p-free-plus-integer}, $G+(\rtp(G)-1)$ is $\mathscr{P}$-free,
    and hence Left's winning move must have outcome $\mathscr{L}$. Again by
    \cref{thm:tipping-points}, observe that $\outcome(G+k)=\mathscr{N}$ for all
    $k$ satisfying $0\leq k\leq\rtp(G)-1$. It then follows that Left's winning
    move on $G+(\rtp(G)-1)$ must be of the form $G^L+(\rtp(G)-1)$. If this
    $G^L$ had outcome $\mathscr{R}$ or $\mathscr{N}$, then so too would
    $G^L+(\rtp(G)-1)$, which is not the case, and hence
    $\outcome(G^L)=\mathscr{L}$. It is then clear that $\rtp(G)-1<\ntp(G^L)$,
    which is equivalent to writing $\ntp(G^L)\geq\rtp(G)$. By \cref{lem:tech1},
    we have the result.
\end{proof}

\begin{lemma}
    \label{prop:ends-tp}
    If $\set{A}$ is an outcome-stable, hereditary monoid containing $1$ and
    $\overline{1}$, and $G\in\pf{A}$ is a Left end, then $\rtp(G)=\ntp(G)+1$.
    (If $G$ is instead a Right end, then $\ltp(G)=\ntp(G)+1$.)
\end{lemma}

\begin{proof}
    Since $G$ is a Left end, we must have that $\outcome(G)\geq\mathscr{N}$. If
    $\outcome(G)=\mathscr{N}$, then we have the result by
    \cref{lem:left-end-n-r1}, so assume that $\outcome(G)=\mathscr{L}$ (and so
    $\ntp(G)\geq1$).

    We know by \cref{thm:tipping-points} that $G+\ntp(G)=\mathscr{N}$. By
    \cref{lem:tech2}, it follows that $\ntp(G+(\ntp(G)-1))=\rtp(G+\ntp(G))$.
    And $\ntp(G+(\ntp(G)-1))=1$, so $\rtp(G+\ntp(G))=1$, yielding the result.
\end{proof}

We make explicit, for completeness, what we already know from the proof of
\cref{tipping-point-existence}: if $G$ is both a Left and a Right end, then $G$
must be the zero form, and hence it follows from \cref{prop:ends-tp} that
$\ltp(G)=\rtp(G)=1$ as clearly $\ntp(G)=0$.

\begin{lemma}[cf. {\cite[Lemmas 8--9 on p.~7]{milley.renault:invertible}}]
    \label{lem:tech3}
    If $\set{A}$ is an outcome-stable, hereditary monoid containing $1$ and
    $\overline{1}$, and $G\in\pf{A}$ with $\outcome(G)=\mathscr{L}$, then we
    have the following:
    \begin{enumerate}
        \item
            if $\ntp(G)\neq\rtp(G)-1$, then there exists some option $G^L$ with
            $\outcome(G^L)=\mathscr{L}$ such that $\ntp(G^L)=\rtp(G)$;
        \item
            if $G^R$ is a Right option of $G$, then $\rtp(G^R)\geq\ntp(G)$; and
        \item
            there exists some $G^R$ such that $\rtp(G^R)=\ntp(G)$.
    \end{enumerate}
\end{lemma}

\begin{proof}
    First we prove (1). By \cref{prop:ends-tp}, we know that $G$ is not a Left
    end. By \cref{thm:tipping-points}, recall that
    $\outcome(G+(\rtp(G)-1))=\mathscr{N}$. By hypothesis,
    $\ntp(G)\neq\rtp(G)-1$, and hence $\outcome(G+(\rtp(G)-2))=\mathscr{N}$. As
    such, Left's winning move in $G+(\rtp(G)-1)$ must be of the form
    $G^L+(\rtp(G)-1)$, which necessarily has outcome $\mathscr{L}$ (recall
    \cref{lem:p-free-plus-integer}). If $G^L$ had outcome $\mathscr{N}$ or
    $\mathscr{R}$, then so too would $G^L+(\rtp(G)-1)$, which is not the case,
    and hence $\outcome(G^L)=\mathscr{L}$. It is then clear that
    $\ntp(G^L)\geq\rtp(G)$, and we are done by \cref{lem:tech1}.

    Now we prove (2). By \cref{thm:tipping-points}, we know that
    $\outcome(G+(\ntp(G)-1))=\mathscr{L}$, and hence Right has no winning
    option. So, if $G^R$ is a Right option of $G$, then $G^R+(\ntp(G)-1)$ must
    have outcome $\mathscr{N}$ or $\mathscr{L}$. It then follows that
    $\rtp(G^R)\geq \ntp(G)$.

    Finally, we prove (3). Since $\outcome(G)=\mathscr{L}$, note that $G$
    cannot be Right end-like. By \cref{thm:tipping-points}, we know that
    $\outcome(G+\ntp(G))=\mathscr{N}$, and hence it follows that there exists a
    Right option $G^R$ of $G$ such that $\outcome(G^R+\ntp(G))=\mathscr{R}$
    (recall \cref{lem:p-free-plus-integer}). We then observe that
    $\rtp(G^R)\leq\ntp(G)$ (by \cref{thm:tipping-points}), and so
    $\rtp(G^R)=\ntp(G)$ by (2).
\end{proof}

A table summarising the results of this subsection can be found in
\cref{tab:techs} in \cref{app:tables}.

\subsection{Integer-invertible sets of games}
\label{sec:int}

\begin{definition}
    \label{def:integral}
    We say a set of games $\set{A}$ is \emph{integer-invertible} if
    $n,\overline{n}\in\set{A}$ and $n+\overline{n}\equiv_\set{A}0$ for all
    integers $n$.
\end{definition}

Note that, if $\set{A}$ is a semigroup of games, then it is integer-invertible
if and only if $1,\overline{1}\in\set{A}$ and $1+\overline{1}\equiv_\set{A}0$.
Furthermore, if $\set{U}$ is a universe (or, more generally, is
conjugate-closed and has the conjugate property), then $\set{U}$ is
integer-invertible if and only if $1\in\set{U}$ and $1$ is
$\set{U}$-invertible.

Whereas being outcome-stable is preserved under taking subsets, being
integer-invertible is \emph{not}. As a straightforward example, consider the
dicot universe $\set{D}$, which is a subuniverse of the dead-ending universe
$\set{E}$. It is known that $\set{E}$ is integer-invertible \cite[Theorem 19 on
p.~11]{milley.renault:invertible}, but $\set{D}$ cannot be integer-invertible
since $1\notin\set{D}$! But this is in fact the only way integer-invertibility
can fail for a subset: a subset of an integer-invertible set is itself
integer-invertible if and only if it contains both $n$ and $\overline{n}$ for
all integers $n$.

Of course, every integer-invertible semigroup of games is a monoid, which
simplifies some of the statements below. (But we remind the reader that the
strict form $\{\cdot\mid\cdot\}$ is not necessarily an element of monoids that
are not hereditary.) To see why this is the case, let $\set{A}$ be an
integer-invertible semigroup of games. We know that
$1+\overline{1}\equiv_\set{A}0$, and also that $1+\overline{1}\in\set{A}$.
Hence, if $G\in\set{A}$, then $1+\overline{1}+G\equiv_\set{A}G$, which yields
that $\set{A}$ is a monoid.

If we are working with an integer-invertible monoid $\set{A}$, then it follows
immediately from the definitions (of \emph{integer-invertible} and
\emph{semigroup}/\emph{monoid}) that if $G,H\in\set{A}$, then
$G+H\equiv_\set{A}G+n+H+\overline{n}$ for all integers $n$. We will make use of
this remark repeatedly below.

\begin{lemma}[cf. {\cite[Lemmas 11--12 on p.~7]{milley.renault:invertible}}]
    \label{lem:11}
    If $\set{A}$ is an outcome-stable and integer-invertible monoid, and
    $G,H\in\pf{A}[A]$ with $\outcome(G)=\mathscr{L}$ and
    $\outcome(H)=\mathscr{N}$, then we have the following:
    \begin{enumerate}
        \item
            if $\ntp(G)>\ltp(H)$, then $\outcome(G+H)=\mathscr{L}$; and
        \item
            if $\rtp(G)<\ltp(H)$, then $\outcome(G+H)=\mathscr{N}$.
    \end{enumerate}
    (The symmetric results hold when $\outcome(G)=\mathscr{R}$ and
    $\outcome(H)=\mathscr{N}$, exchanging $\mathscr{L}$- and
    $\mathscr{R}$-tipping points.)
\end{lemma}

\begin{proof}
    We may assume $G,H\in\pf{A}$. First we prove (1). Observe that
    \begin{align*}
        \outcome(G+H)&=
        \outcome(\underbrace{G+\ltp(H)}_{\mathscr{L}}+
        \underbrace{H+\overline{\ltp(H)}}_{\mathscr{L}})\\
                     &=\mathscr{L}
    \end{align*}
    since $\set{A}$ is outcome-stable and integer-invertible.

    Now we prove (2). Observe that
    \begin{align*}
        \outcome(G+H)&=
        \outcome(\underbrace{G+\rtp(G)}_{\mathscr{R}}+
        \underbrace{H+\overline{\rtp(G)}}_{\mathscr{N}})\\
                     &\leq\mathscr{N}.
    \end{align*}
    But we know already that $\outcome(G+H)\geq\mathscr{N}$ since $\set{A}$ is
    outcome-stable, and hence $\outcome(G+H)=\mathscr{N}$.
\end{proof}

\begin{lemma}[cf. {\cite[Lemma 13 on p.~8]{milley.renault:invertible}}]
    \label{lem:13}
    If $\set{A}$ is an outcome-stable and integer-invertible monoid, and
    $G,H\in\pf{A}[A]$ with $\outcome(G),\outcome(H)=\mathscr{N}$, then
    \begin{align*}
        \mathrm{if~}\rtp(G)>\ltp(H)\mathrm{~or~}\ltp(G)<\rtp(H)\mathrm{,~then}
        \outcome(G+H)\geq\mathscr{N}.
    \end{align*}
    (The symmetric result holds for $\outcome(G+H)\leq\mathscr{N}$, reversing
    tipping point inequality conditions.)
\end{lemma}

\begin{proof}
    We may assume $G,H\in\pf{A}$. Suppose that $\rtp(G)>\ltp(H)$. Observe that
    \begin{align*}
        \outcome(G+H)&=
        \outcome(\underbrace{G+\ltp(H)}_{\mathscr{N}}+
        \underbrace{H+\overline{\ltp(H)}}_{\mathscr{L}})\\
                     &\geq\mathscr{N}
    \end{align*}
    since $\set{A}$ is outcome-stable and integer-invertible.

    Now suppose that $\ltp(G)<\rtp(H)$. Observe that
    \begin{align*}
        \outcome(G+H)&=
        \outcome(\underbrace{G+\overline{\ltp(G)}}_{\mathscr{L}}+
        \underbrace{H+\ltp(G)}_{\mathscr{N}})\\
                     &\geq\mathscr{N}
    \end{align*}
    since $\set{A}$ is outcome-stable and integer-invertible.
\end{proof}

\begin{lemma}[cf. {\cite[Lemma 14 on p.~8]{milley.renault:invertible}}]
    \label{lem:14}
    If $\set{A}$ is an outcome-stable and integer-invertible monoid of games,
    and $G,H\in\pf{A}[A]$ with $\outcome(G)=\mathscr{L}$ and
    $\outcome(H)=\mathscr{R}$, then we have the following:
    \begin{enumerate}
        \item
            if $\ntp(G)>\ltp(H)$, then $\outcome(G+H)=\mathscr{L}$;
        \item
            if $\ntp(G)>\ntp(H)$ or $\rtp(G)>\ltp(H)$, then
            $\outcome(G+H)\geq\mathscr{N}$;
        \item
            if $\ntp(G)>\ntp(H)$ and $\rtp(G)<\ltp(H)$, then
            $\outcome(G+H)=\mathscr{N}$;
            and
        \item
            if $\ntp(G)<\ntp(H)$ and $\rtp(G)>\ltp(H)$, then
            $\outcome(G+H)=\mathscr{N}$.
    \end{enumerate}
    (The symmetries of (1) and (2) hold for $\outcome(G+H)=\mathscr{R}$ and
    $\outcome(G+H)\leq\mathscr{N}$.)
\end{lemma}

\begin{proof}
    We may assume $G,H\in\pf{A}$. First we prove (1). Observe that
    \begin{align*}
        \outcome(G+H)&=
        \outcome(\underbrace{G+(\ntp(G)-1)}_{\mathscr{L}}+
        \underbrace{H+\overline{\ntp(G)-1}}_{\mathscr{L}})\\
                     &=\mathscr{L}
    \end{align*}
    since $\set{A}$ is outcome-stable and integer-invertible.

    Now we prove (2). Suppose that $\ntp(G)>\ntp(H)$. Observe that
    \begin{align*}
        \outcome(G+H)&=
        \outcome(\underbrace{G+(\ntp(G)-1)}_{\mathscr{L}}+
        \underbrace{H+\overline{\ntp(G)-1}}_{\geq\mathscr{N}})\\
                     &\geq\mathscr{N}
    \end{align*}
    since $\set{A}$ is outcome-stable and integer-invertible.

    Now suppose that $\rtp(G)>\ltp(H)$. Observe that
    \begin{align*}
        \outcome(G+H)&=
        \outcome(\underbrace{G+(\rtp(G)-1)}_{\mathscr{N}}+
        \underbrace{H+\overline{\rtp(G)-1}}_{\mathscr{L}})\\
                     &\geq\mathscr{N}
    \end{align*}
    since $\set{A}$ is outcome-stable and integer-invertible.

    Observe that (3) and (4) follow swiftly by the symmetries of the previous
    case, as both $\outcome(G+H)\geq\mathscr{N}$ and
    $\outcome(G+H)\leq\mathscr{N}$ hold.
\end{proof}

A table summarising the results of this subsection can be found in
\cref{tab:11-14} in \cref{app:tables}.

\subsection{Putting it all together}
\label{sec:combo}

We have gathered most of the tools needed now to prove our main result for
particular sets of $\mathscr{P}$-free games being additively closed. We turn to
a final property that we have not yet been able to eliminate, needed to define
the monoids of interest for our study, before we cover the remaining cases and
bring together the waterfall of lemmas stated thus far.

\begin{definition}[Property X]
    \hypertarget{prop:X}
    If $\set{A}$ is a set of games satisfying the following properties for all
    $G,H\in\pf{A}$ with $\outcome(G),\outcome(H)=\mathscr{N}$, then we say
    $\set{A}$ has \emph{Property X}:
    \begin{enumerate}
        \item
            if $\rtp(G)=\ltp(H)=1$, where $G$ is Left end-like and $H$ is not,
            then $\outcome(G+H)=\mathscr{N}$; and
        \item
            if $\ltp(G)=\rtp(H)=1$, where $H$ is Right end-like and $G$ is not,
            then $\outcome(G+H)=\mathscr{N}$.
    \end{enumerate}
\end{definition}

Note that condition (2) in the above is the symmetric statement to (1). Thus,
in a conjugate-closed monoid, one need only check one of the conditions. Also,
a subset of a set of games satisfying \hyperlink{prop:X}{Property X} must
itself also satisfy \hyperlink{prop:X}{Property X}.

There was no need for such a \hyperlink{prop:X}{Property X} in Milley and
Renault's paper \cite{milley.renault:invertible} on the dead-ending universe
specifically, but we must be careful in this setting of greater generality: in
$\set{E}$, the only end with outcome $\mathscr{N}$ is 0, and so this
\hyperlink{prop:X}{Property X} would be trivially satisfied. We will see later
in \cref{lem:B-prop-X} that the blocking universe $\set{B}$ also has
\hyperlink{prop:X}{Property X}, although it is not as trivial as it was for
$\set{E}$: for example, the Left blocked end $\{\cdot\mid0,1\}$ has outcome
$\mathscr{N}$ but is not equal to 0 (they are distinguished by the game $*$).

\begin{lemma}[cf. {\cite[Theorem 15 on pp.~8--9]{milley.renault:invertible}}]
    \label{thm:final-piece}
    If $\set{A}$ is an outcome-stable, hereditary, and integer-invertible
    monoid that has \hyperlink{prop:X}{Property X}, and $G,H\in\pf{A}[A]$, then
    we have the following:
    \begin{enumerate}
        \item
            if $\outcome(G)=\mathscr{L}$ and $\outcome(H)=\mathscr{N}$, then
            \begin{enumerate}
                \item
                    if $\ntp(G)=\ltp(H)$, then $\outcome(G+H)=\mathscr{L}$; and
                \item
                    if $\rtp(G)=\ltp(H)$, then $\outcome(G+H)=\mathscr{N}$;
            \end{enumerate}
        \item 
            if $\outcome(G)=\mathscr{R}$ and $\outcome(H)=\mathscr{N}$, then
            \begin{enumerate}
                \item
                    if $\ntp(G)=\rtp(H)$, then $\outcome(G+H)=\mathscr{R}$; and
                \item
                    if $\ltp(G)=\rtp(H)$, then $\outcome(G+H)=\mathscr{N}$;
            \end{enumerate}
        \item
            if $\outcome(G),\outcome(H)=\mathscr{N}$ and if $\rtp(G)=\ltp(H)$
            or $\ltp(G)=\rtp(H)$, then $\outcome(G+H)=\mathscr{N}$; and
        \item
            if $\outcome(G)=\mathscr{L}$ and $\outcome(H)=\mathscr{R}$, then
            \begin{enumerate}
                \item
                    if $\ntp(G)=\ltp(H)$, then $\outcome(G+H)=\mathscr{L}$;
                \item
                    if $\ntp(G)=\ntp(H)$ or $\rtp(G)=\ltp(H)$, then
                    $\outcome(G+H)=\mathscr{N}$; and
                \item
                    if $\rtp(G)=\ntp(H)$, then $\outcome(G+H)=\mathscr{R}$.
            \end{enumerate}
    \end{enumerate}
\end{lemma}

\begin{proof}
    We may assume $G,H\in\pf{A}$. It is straightforward to verify the result in
    the case that either $G$ or $H$ is equal to 0, so we now suppose that $G$
    and $H$ are non-zero, and we assume by induction that the result holds for
    all pairs $(G',H')$ where $\birth(G'+H')<\birth(G+H)$.

    \begin{enumerate}
        \item Assume that $\outcome(G)=\mathscr{L}$ and
            $\outcome(H)=\mathscr{N}$. Starting with (a), we assume further
            that $\ntp(G)=\ltp(H)$. We will show that
            $\outcome(G+H)=\mathscr{L}$.

            Since $\set{A}$ is outcome-stable, we know that
            $\outcome(G+H)\geq\mathscr{N}$. Note that $G+H$ cannot be Right
            end-like, since that would imply $G$ and $H$ were both Right
            end-like, contradicting $G$ having outcome $\mathscr{L}$. Suppose,
            for a contradiction, that $\outcomeR(G+H)=\mathscr{R}$; i.e.\
            either $G+H$ has a Right tombstone, or else Right has a winning
            move. Since $G+H$ cannot be Right end-like, we know that Right has
            a winning move. Because $\set{A}$ is outcome-stable, such a move
            must lie in the following cases:
            \begin{enumerate}[i)]
                \item
                    $G^R+H$, where $\outcome(G^R),\outcome(H)=\mathscr{N}$

                    Since $\outcome(G)=\mathscr{L}$, we have by
                    \cref{lem:tech3} that $\rtp(G^R)\geq\ntp(G)$. Since
                    $\ntp(G)=\ltp(H)$ by supposition, we have that
                    $\rtp(G^R)\geq\ltp(H)$. If $\rtp(G^R)>\ltp(H)$, then
                    $\outcome(G^R+H)\geq\mathscr{N}$ by \cref{lem:13}.
                    Otherwise, if $\rtp(G^R)=\ltp(H)$, then
                    $\outcome(G^R+H)=\mathscr{N}$ by induction, so this cannot
                    be a winning option for Right.
                \item
                    $G+H^R$, where $\outcome(G)=\mathscr{L}$ and
                    $\outcome(H^R)=\mathscr{R}$

                    Since $\outcome(H)=\mathscr{N}$, we know by the symmetric
                    result of \cref{lem:tech1} that $\ntp(H^R)\leq\ltp(H)$ for
                    all Right options $H^R$ of $H$ with
                    $\outcome(H^R)\neq\mathscr{L}$. Since $\ntp(G)=\ltp(H)$ by
                    supposition, we have that $\ntp(H^R)\leq\ntp(G)$. If
                    $\ntp(H^R)<\ntp(G)$, then $\outcome(G+H^R)\geq\mathscr{N}$
                    by \cref{lem:14}. Otherwise, if $\ntp(H^R)=\ntp(G)$, then
                    $\outcome(G+H^R)=\mathscr{N}$ by induction, so this also
                    cannot be a winning option for Right. This completes the
                    proof of 1(a).
            \end{enumerate}

            We now proceed with (b). Recall that $\outcome(G)=\mathscr{L}$ and
            $\outcome(H)=\mathscr{N}$, and assume further that
            $\rtp(G)=\ltp(H)$. We will show that $\outcome(G+H)=\mathscr{N}$.

            We will show that Right has a winning option on $G+H$. First we
            demonstrate that $H$ cannot be a Right end. It follows from
            \cref{thm:tipping-points} that $\rtp(G)\geq2$ (since
            $\outcome(G)=\mathscr{L}$). If $H$ were a Right end, then
            $\ltp(H)=1$ by \cref{prop:ends-tp}. This contradicts our
            supposition that $\ltp(H)=\rtp(G)$, and hence $H$ is not a Right
            end.

            Since $\ltp(H)\neq1$, we know by the symmetric result of
            \cref{lem:tech2} that there must exist some option $H^R$ with
            $\outcome(H^R)=\mathscr{R}$ and $\ntp(H^R)=\ltp(H)$. Since
            $\rtp(G)=\ltp(H)$ by supposition, we have that $\ntp(H^R)=\rtp(G)$.
            And now we observe that $\outcome(G+H^R)=\mathscr{R}$ by induction.
            Thus, Right has a good option, and so
            $\outcome(G+H)\leq\mathscr{N}$. But we already know that
            $\outcome(G+H)\geq\mathscr{N}$ since $\set{A}$ is outcome-stable,
            and hence we have that $\outcome(G+H)=\mathscr{N}$ when
            $\rtp(G)=\ltp(H)$, as desired.

        \item It is clear that (2) is symmetric to (1).

        \item Assume $\outcome(G),\outcome(H)=\mathscr{N}$, and further that
            either $\rtp(G)=\ltp(H)$ or $\ltp(G)=\rtp(H)$. We will show that
            $\outcome(G+H)=\mathscr{N}$.

            By symmetry, we need only show that Left wins going first when
            $\rtp(G)=\ltp(H)$. If $G$ and $H$ are both Left end-like, then
            $G+H$ is also Left end-like, and hence Left wins going first on
            $G+H$. If $\rtp(G)=1$ and $G$ is Left end-like but $H$ is not, then
            it follows immediately by $\set{A}$ having
            \hyperlink{prop:X}{Property X} (by hypothesis) that
            $\outcome(G+H)=\mathscr{N}$. Otherwise, we know by \cref{lem:tech2}
            that there must exist some $G^L$ with $\outcome(G^L)=\mathscr{L}$
            and $\ntp(G^L)=\rtp(G)$. Since $\rtp(G)=\ltp(H)$ by supposition, we
            have that $\ntp(G^L)=\ltp(H)$. We then observe that
            $\outcome(G^L+H)=\mathscr{L}$ by induction.

        \item Assume $\outcome(G)=\mathscr{L}$ and $\outcome(H)=\mathscr{R}$.

            Note that $G+H$ cannot be Right end-like, since that would imply
            $G$ and $H$ were both Right end-like, contradicting $G$ having
            outcome $\mathscr{L}$.

            Starting with (a), we assume further that $\ntp(G)=\ltp(H)$. We
            will show that $\outcome(G+H)=\mathscr{L}$.

            It is a simple observation from \cref{thm:tipping-points} that
            $\ltp(H)>\ntp(H)$, and hence $\ntp(G)>\ntp(H)$ here. Thus, by
            \cref{lem:14}, we must have that $\outcome(G+H)\geq\mathscr{N}$,
            and we need only show that Right has no winning move. Since
            $\set{A}$ is outcome-stable, a winning move for Right must fall
            into the following cases:
            \begin{enumerate}[i)]
                \item
                    $G^R+H$, where $\outcome(G^R)=\mathscr{L}$ and
                    $\outcome(H)=\mathscr{R}$;

                    \cref{lem:tech3} yields that $\rtp(G^R)\geq\ntp(G)$ for all
                    options $G^R$, meaning $\rtp(G^R)\geq \ltp(H)$ here. If
                    $\rtp(G^R)>\ltp(H)$, then $\outcome(G^R+H)\geq\mathscr{N}$
                    by \cref{lem:14}. Otherwise, if $\rtp(G^R)=\ltp(H)$, then
                    $\outcome(G^R+H)=\mathscr{N}$ by induction, so this is not
                    a winning option for Right.
                \item
                    $G^R+H$, where $\outcome(G^R)=\mathscr{N}$ and
                    $\outcome(H)=\mathscr{R}$;

                    \cref{lem:tech3} again yields that
                    $\rtp(G^R)\geq\ntp(G)=\ltp(H)$. If $\rtp(G^R)>\ltp(H)$,
                    then $\outcome(G^R+H)=\mathscr{N}$ by the symmetric result
                    of \cref{lem:11}. Otherwise, if $\rtp(G^R)=\ltp(H)$, then
                    $\outcome(G^R+H)=\mathscr{N}$ by induction, so this is not
                    a winning option for Right.
                \item
                    $G+H^R$, where $\outcome(G)=\mathscr{L}$, and
                    $\outcome(H^R)=\mathscr{R}$.

                    By the symmetry of \cref{lem:tech1}, we have that
                    $\ntp(H^R)\leq\ltp(H)$ for all Right options $H^R$ with
                    $\outcome(H^R)\neq\mathscr{L}$. Since $\ntp(G)=\ltp(H)$ by
                    supposition, we have that $\ntp(H^R)\leq\ntp(G)$. If
                    $\ntp(H^R)<\ntp(G)$, then $\outcome(G+H^R)=\mathscr{N}$ by
                    \cref{lem:14}. Otherwise, if $\ntp(H^R)=\ntp(G)$, then
                    $\outcome(G+H^R)=\mathscr{N}$ by induction, so this is not
                    a winning option for Right.
            \end{enumerate}

            We have shown that Right has no good options when
            $\ntp(G)=\ltp(H)$, and so $\outcome(G+H)=\mathscr{L}$ here,
            completing the proof of (a).

            For (b), recall that $\outcome(G)=\mathscr{L}$ and
            $\outcome(H)=\mathscr{R}$, and further that either
            $\ntp(G)=\ntp(H)$ or $\rtp(G)=\ltp(H)$. We will show that
            $\outcome(G+H)=\mathscr{N}$. By symmetry, we need only show that
            Left has a winning move in each case.

            If $\ntp(G)=\ntp(H)$, by the symmetry of \cref{lem:tech3}, there
            must exist some option $H^L$ with $\ltp(H^L)=\ntp(H)=\ntp(G)$. Then
            $\outcome(G+H^L)=\mathscr{L}$ by induction, so Left can win on
            $G+H$.

            For $\rtp(G)=\ltp(H)$, we must consider two cases:
            \begin{enumerate}[i)]
                \item $\ntp(G)=\rtp(G)-1$;

                    By \cref{lem:tech3}'s symmetry, there must exist some
                    option $H^L$ such that $\ltp(H^L)=\ntp(H)$. But now since
                    $\ntp(H)<\ltp(H)=\rtp(G)$ by supposition, it follows that
                    $\ltp(H^L)<\rtp(G)$, and so also $\ltp(H^L)\leq\ntp(G)$. We
                    then observe that $\outcome(G+H^L)=\mathscr{L}$ by
                    induction with either \cref{lem:11} or \cref{lem:14}. Thus,
                    Left has a winning option in this first case.

                \item $\ntp(G)\neq\rtp(G)-1$.

                    By \cref{prop:ends-tp}, we know that $G$ is not a Left end.
                    By \cref{lem:tech3}, there must exist some option $G^L$
                    with $\outcome(G^L)=\mathscr{L}$ and $\ntp(G^L)=\rtp(G)$,
                    which then yields $\ntp(G^L)=\ltp(H)$. Therefore,
                    $\outcome(G^L+H)=\mathscr{L}$ by induction, ensuring Left
                    may win in this case as well.
            \end{enumerate}

            Observe that (c) is symmetric to (a), and so we are done.
    \end{enumerate}
\end{proof}

We have had to equip $\set{A}$ with \hyperlink{prop:X}{Property X} in the
hypothesis of \cref{thm:final-piece} due only to the difficulty of the case
(and its symmetric formulation) where $\outcome(G),\outcome(H)=\mathscr{N}$,
and $G$ is Left end-like and $H$ is not, and $\rtp(G)=\ltp(H)=1$. This is
because of examples such as the following: taking $G=\{\cdot\mid2\}$ and
$H=\{\overline{1}\mid\cdot\}$, it is clear that
$\outcome(G)=\outcome(H)=\mathscr{N}$ and $\rtp(G)=\ltp(H)=1$, with $G$ a Left
end and $H$ not, but $\outcome(G+H)=\mathscr{R}\neq\mathscr{N}$. But such an
example cannot occur in an integer-invertible monoid.\footnote{Davies and Yadav
    showed \cite[Theorem 5.9 on p.~23]{davies.yadav:invertibility} that every
    invertible element of a monoid containing a game such as $\{\cdot\mid2\}$
    (which is called \emph{Left weak}) must be a Left end, and hence 1 is not
    invertible modulo such a monoid, meaning it cannot be integer-invertible.
}
It is not clear whether one can use the other properties in the hypothesis to
exclude all such counter-examples and thus eliminate the need for
\hyperlink{prop:X}{Property X}. (Recall that, in Milley and Renault's proof,
such things need not be considered since it is \emph{trivial} that $\set{E}$
satisfies \hyperlink{prop:X}{Property X}.)

Yet a different pattern exists for the form $G$ given by
\cref{fig:strange-form},
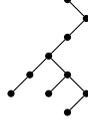
\begin{figure}
    \centering
    \begin{tikzpicture}[scale=0.5]
        \tikzset{dot/.style={circle,draw,fill,inner sep=.8pt}}
        \node[dot] (n0) at (1.5,0) {};
        \node[dot] (n1) at (2,-0.5) {};
        \node[dot] (n2) at (1.5,-1) {};
        \node[dot] (n3) at (1,-1.5) {};
        \node[dot] (n4a) at (.5,-2) {};
        \node[dot] (n4b) at (1.5,-2) {};
        \node[dot] (n5a) at (0,-2.5) {};
        \node[dot] (n5b) at (1,-2.5) {};
        \node[dot] (n5c) at (2,-2.5) {};
        \node[dot] (n6) at (1.5,-3) {};
        \draw(n0)--(n1)--(n2)--(n3)--(n4a)--(n5a);
        \draw(n3)--(n4b);
        \draw(n5b)--(n4b)--(n5c)--(n6);
    \end{tikzpicture}
    \caption{
        The form
        $\left\{\cdot\mid\{\{\{1\mid\{0\mid1\}\}\mid\cdot\}\mid\cdot\}\right\}$.
    }
    \label{fig:strange-form}
\end{figure}
where $\outcome(G)=\mathscr{N}$ and $G$ is $\mathscr{P}$-free. Note that $G$ is
not in the blocking universe, and so will not be covered by our later
\cref{lem:B-outcome-stable}. When added to $H=\{\overline{1}\mid1\}$, we see
$\rtp(G)=\ltp(H)=1$, yet $\outcome(G+H)=\mathscr{R}$. Thus, no set $\set{A}$
containing $G$ and $H$ can have \hyperlink{prop:X}{Property X}. But, like
before, this is overshadowed by the failure of $\set{A}$ to be an
integer-invertible monoid, as $1+\overline{1}+G\not\equiv_\set{A}G$ (simply
compare outcomes). Thus, again, the necessity of \hyperlink{prop:X}{Property X}
remains an open question (see \cref{open:X}).

We state the culmination of what we have shown thus far. We summarize how all
the previous lemmas build up to this result in \cref{tab:11-14,tab:combo} in
\cref{app:tables}.

\begin{lemma}
    \label{thm:pfree-sum}
    If $\set{A}$ is an outcome-stable, hereditary, and integer-invertible
    monoid with \hyperlink{prop:X}{Property X}, and $G,H\in\pf{A}[A]$, then
    $\outcome(G+H)\neq\mathscr{P}$.
\end{lemma}

\begin{proof}
    \cref{lem:11,lem:13,lem:14,thm:final-piece} combine to yield the result.
\end{proof}

This now lets us easily prove our main result of the section.

\begin{theorem}
    \label{thm:p-free-closed}
    If $\set{A}$ is an outcome-stable, hereditary, and integer-invertible
    monoid with \hyperlink{prop:X}{Property X}, then $\pf{A}$ is a monoid.
\end{theorem}

\begin{proof}
    Let $G,H\in\pf{A}$. Every option of $G+H$ is necessarily of the form $G'+H$
    or $G+H'$, where $G',H'\in\pf{A}$. Thus, we must have $G'+H,G+H'\in\pf{A}$
    by induction. By \cref{thm:pfree-sum}, we know that
    $\outcome(G+H)\neq\mathscr{P}$, and hence we have the result.
\end{proof}

It will sometimes be more convenient to use the following weaker corollary.

\begin{corollary}
    If $\set{A}$ is an outcome-stable, hereditary, and integer-invertible
    monoid that has \hyperlink{prop:X}{Property X}, then $\pf{A}[A]$ is a
    monoid.
\end{corollary}

\begin{proof}
    Let $G,H\in\pf{A}[A]$. By definition, there must exist games
    $G',H'\in\pf{A}$ with $G'\equiv_\set{A}G$ and $H'\equiv_\set{A}H$. By
    \cref{thm:p-free-closed}, we have $G'+H'\in\pf{A}$. Since $\set{A}$ is a
    monoid, it follows that $G+H\equiv_\set{A}G'+H'$, and hence $G+H$ is
    $\mathscr{P}$-free modulo $\set{A}$.
\end{proof}

\begin{corollary}
    \label{thm:p-free-semigroup}
    If $\set{A}$ is a subsemigroup of an outcome-stable, hereditary, and
    integer-invertible monoid that has \hyperlink{prop:X}{Property X}, then
    either $\pf{A}$ is a semigroup or else $\pf{A}=\emptyset$ (and similarly
    for $\pf{A}[A]$).
\end{corollary}

\begin{proof}
    Suppose that $\pf{A}\neq\emptyset$. Let $\set{S}\supseteq\set{A}$ be an
    outcome-stable, hereditary, and integer-invertible monoid with
    \hyperlink{prop:X}{Property X}, and let $G,H\in\pf{A}$. Since
    $\set{A}\subseteq\set{S}$, we know that $G,H\in\set{S}$. It then follows
    from \cref{thm:p-free-closed} that $G+H\in\pf{A}$. Since $\set{A}$ is
    additively closed, we obtain $G+H\in\pf{A}$, and hence also the result.
\end{proof}

To see why we need to take care that $\pf{A}$ may be empty, even when $\set{A}$
is not, consider the set of games $\set{A}=\{n\cdot*:n\in\mathbb{N}_1\}$. It is
clear that $(\set{A},\equiv_\set{A})$ is a semigroup. Observe that the
dead-ending universe $\set{E}$ is an outcome-stable, hereditary, and
integer-invertible monoid that has Property $X$ and contains $\set{A}$. But we
can also observe that $\pf{A}=\emptyset$, since $*$, which has outcome
$\mathscr{P}$, is a subposition of every element of $\set{A}$. (In fact,
$\set{A}$ is isomorphic to the group $\mathbb{Z}_2$, and it is
conjugate-closed, so it will also serve as a good example for why we have to
take care that $\pf{A}$ may be empty in some of the later results, too.)

\begin{problem}
    \label{open:X}
    Is \hyperlink{prop:X}{Property X} necessary for \cref{thm:p-free-closed}
    and its \cref{thm:p-free-semigroup}?
\end{problem}

\section{Blocking games: an application}
\label{sec:blocking}

The theory of the previous section would not be much use if we could not apply
it to anything! In this section, we will show that our generalisation of Milley
and Renault's arguments does yield something both stronger and practical. In
particular, we will show that the set of $\mathscr{P}$-free blocking games is
closed under addition (i.e.\ that $\pf{B}$ is a monoid). To do so, we will show
that $\set{B}$ satisfies the hypothesis of \cref{thm:p-free-closed}.

Recall from \cite{davies.mckay.ea:pocancellation} that a Left end $X$ is
\emph{blocked} if for every Right option $X^R$ of $X$, either $X^R$ is a
blocked Left end, or else there exists a Left response $X^{RL}$ that is again a
blocked Left end. A Right end being blocked is defined symmetrically. A game is
then called \emph{blocking} if every subposition that is an end is blocked. The
blocking universe $\set{B}$ is defined as the set of all blocking games (these
are ordinary, not augmented, forms).

We start with proving $\set{B}$ is outcome-stable: the following lemma gives
the base case where one of the summands is a Left end.

\begin{lemma}
    \label{lem:L+end}
    If $G,H\in\pf{B}$, where $\outcome(G)=\mathscr{L}$ and $H$ is a Left end,
    then $\outcome(G+H)=\mathscr{L}$.
\end{lemma}

\begin{proof}
    Consider Left playing first on $G+H$. If $G$ is a Left end, then Left
    trivially wins first on $G+H$. If $G$ is not a Left end, then, since
    $\outcome(G)=\mathscr{L}$ and $G\in\pf{B}$, there must exist a Left option
    $G^L$ with $\outcome(G^L)=\mathscr{L}$. Now $\outcome(G^L+H)=\mathscr{L}$
    by induction, and so Left wins $G+H$ playing first. It remains to show that
    Left wins playing second.

    When Right plays first, there are four possibilities:
    \begin{enumerate}
        \item
            if Right moves to $G^R+H$ where $G^R$ is a Left end, then Left wins
            immediately;
        \item
            if Right moves to $G^R+H$ where $G^R$ is not a Left end, then,
            since $\outcome(G)=\mathscr{L}$ and $G\in\pf{B}$, there is a
            Left-win $G^{RL}$, and so Left wins by induction moving to
            $G^{RL}+H$;
        \item
            if Right moves to $G+H^R$ where $H^R$ is a Left end, then Left wins
            by induction; and
        \item
            if Right moves to $G+H^R$ where $H^R$ is not a Left end, then,
            since $H\in\set{B}$, there exists a Left response $H^{RL}$ which is
            a Left end, and so Left wins by induction moving to $G+H^{RL}$.
    \end{enumerate}
\end{proof}

\begin{lemma}
    \label{lem:B-outcome-stable}
    The blocking universe $\set{B}$ is outcome-stable.
\end{lemma}

\begin{proof}
    Let $G,H\in\pf{B}$. By symmetry, we may assume that
    $\outcome(G)=\mathscr{L}$. We need to consider two cases: when
    $\outcome(H)=\mathscr{L}$; and when $\outcome(H)=\mathscr{N}$.

    First, suppose that $\outcome(H)=\mathscr{L}$. We need to show that
    $\outcome(G+H)=\mathscr{L}$. If at least one of $G$ and $H$ is a Left end,
    then we know that $\outcome(G+H)=\mathscr{L}$ by \cref{lem:L+end}. As such,
    assume that neither $G$ nor $H$ is a Left end. Since $G$ is
    $\mathscr{P}$-free, we know that there exists some Left option $G^L$ with
    $\outcome(G^L)=\mathscr{L}$. By induction, it is clear that
    $\outcome(G^L+H)=\mathscr{L}$, and hence $\outcomeL(G+H)=\mathscr{L}$. It
    remains to show that $\outcomeR(G+H)=\mathscr{L}$. We may assume, without
    loss of generality, that Right moves to some $G+H^R$. If
    $\outcome(H^R)=\mathscr{L}$, then $\outcome(G+H^R)=\mathscr{L}$ by
    induction. Otherwise, if $\outcome(H^R)=\mathscr{N}$, then either $H^R$ is
    a Left end, in which case $\outcome(G+H^R)=\mathscr{L}$ by
    \cref{lem:L+end}, or else there exists some $H^{RL}$ with
    $\outcome(H^{RL})=\mathscr{L}$, in which case
    $\outcome(G+H^{RL})=\mathscr{L}$ by induction. Thus,
    $\outcomeR(G+H)=\mathscr{L}$.

    We now suppose that $\outcome(H)=\mathscr{N}$. We must show Left wins
    playing first on $G+H$. If $H$ is a Left end, then \cref{lem:L+end} yields
    the result. So, assume $H$ is not a Left end. Since $H$ is
    $\mathscr{P}$-free, there must exist some $H^L$ with
    $\outcome(H^L)=\mathscr{L}$. We know that $\outcome(G+H^L)=\mathscr{L}$ by
    the previous paragraph, and hence we are done.
\end{proof}

It is clear that all integers $n$ and $\overline{n}$ are blocking forms. To
show that $\set{B}$ is integer-invertible, we must show additionally that every
integer is $\set{B}$-invertible. Before we can prove it, we need to recall the
test proved in \cite{davies.milley:order} for being Left $\set{B}$-strong.

\begin{theorem}[{\cite[Theorem 3.1 on p.~9]{davies.milley:order}}]
    \label{thm:strong}
    If $G\in\set{M}$, then $G$ is Left $\set{B}$-strong if and only if either
    \begin{enumerate}
        \item
            $G$ is a Left end; or
        \item
            $G$ has a Left-win option $G^L$ such that $G^L$ and all Right
            responses $G^{LR}$ are Left $\set{B}$-strong.
    \end{enumerate}
\end{theorem}

\begin{lemma}
    \label{thm:B-pfree-strong}
    If $G\in\pf{B}[B]$ and $\outcome(G)\neq\mathscr{R}$, then $G$ is Left
    $\set{B}$-strong.
\end{lemma}

\begin{proof}
    We may assume $G\in\pf{B}$. We proceed according to \cref{thm:strong}. If
    $G$ is a Left end, then we are done. So, assume that $G$ is not a Left end.
    Since $G$ is $\mathscr{P}$-free and $\outcome(G)\neq\mathscr{R}$, we know
    that there exists some Left option $G^L$ with $\outcome(G^L)=\mathscr{L}$.
    Since every subposition of $G$ must be $\mathscr{P}$-free, we know by
    induction that $G^L$ is Left $\set{B}$-strong. Since
    $\outcome(G^L)=\mathscr{L}$, we know that $\outcome(G^{LR})\geq\mathscr{N}$
    for all Right options $G^{LR}$. Hence, again by induction, we know that
    every $G^{LR}$ is Left $\set{B}$-strong. Thus, by \cref{thm:strong}, we may
    conclude that $G$ is Left $\set{B}$-strong.
\end{proof}

\begin{lemma}
    \label{lem:B-int-invertible}
    The blocking universe $\set{B}$ is integer-invertible.
\end{lemma}

\begin{proof}
    We know that $n,\overline{n}\in\set{B}$ for all $n$. It remains to show
    that $n+\overline{n}\equiv_\set{B}0$. Following \cref{thm:comparison}, the
    maintenance is satisfied by induction, and the proviso is given by
    \cref{thm:B-pfree-strong}.
\end{proof}

Complementary to \cref{lem:left-end-n-r1}, it is interesting to note that a
$\mathscr{P}$-free blocked Left end $G$ with outcome $\mathscr{N}$ must also
have $\ltp(G)=1$, although we will not need it here.

\begin{proposition}
    If $G\in\pf{B}$ is a Left end with $\outcome(G)=\mathscr{N}$, then
    $\ltp(G)=1$.
\end{proposition}

\begin{proof}
    It is clear that $\outcome(G+\overline{1})\geq\mathscr{N}$. Playing first
    on $G+\overline{1}$, if Right moves to $G$, then Left wins immediately.
    Otherwise, Right plays to some $G^R+\overline{1}$. If $G^R$ is a Left end,
    then Left wins immediately. Otherwise, since $G$ is a blocked Left end,
    Left can play on $G^R$ to a blocked Left end $G^{RL}$, and hence wins by
    induction. Thus, $\outcome(G+\overline{1})=\mathscr{L}$, yielding the
    result.
\end{proof}

To prove that $\set{B}$ has \hyperlink{prop:X}{Property X}, we will first
establish a short lemma.

\begin{lemma}
    \label{lem:left-end-plus-N}
    If $G,H\in\pf{B}$, where $G$ is a Left end with
    $\outcome(G),\outcome(H)=\mathscr{N}$, then $\outcomeR(G+H)=\mathscr{R}$.
\end{lemma}

\begin{proof}
    If $G$ is a Right end, then $G\cong0$ and the conclusion is immediate.
    Otherwise, there exists some $G^R$ with $\outcome(G^R)=\mathscr{R}$. If $H$
    is a Right end, then $\outcome(G^R+H)=\mathscr{R}$ by the symmetry of
    \cref{lem:L+end}, which gives the result. Suppose now that neither $G$ nor
    $H$ is a Right end.

    There must exist some $H^R$ with $\outcome(H^R)=\mathscr{R}$. Since $G$ is
    a Left end, Left must respond to $G+H^R$ with an option of the form
    $G+H^{RL}$, where $\outcome(H^{RL})\leq\mathscr{N}$. If
    $\outcome(H^{RL})=\mathscr{R}$, then we are done since $\set{B}$ is
    outcome-stable (\cref{lem:B-outcome-stable}). Otherwise, if
    $\outcome(H^{RL})=\mathscr{N}$, then $\outcomeR(G+H^{RL})=\mathscr{R}$ by
    induction, and hence $\outcomeR(G+H)=\mathscr{R}$.
\end{proof}

\begin{lemma}
    \label{lem:B-prop-X}
    The blocking universe has \hyperlink{prop:X}{Property X}.
\end{lemma}

\begin{proof}
    Let $G,H\in\pf{B}$ with $\outcome(G),\outcome(H)=\mathscr{N}$. Furthermore,
    suppose without loss of generality that $\rtp(G)=\ltp(H)=1$ and that $G$ is
    a Left end and $H$ is not. We will show that $\outcome(G+H)=\mathscr{N}$.

    Since $\outcome(H)=\mathscr{N}$, it follows that there exists some option
    $H^L$ with $\outcome(H^L)=\mathscr{L}$. By \cref{lem:L+end},
    $\outcome(G+H^L)=\mathscr{L}$, and hence $\outcomeL(G+H)=\mathscr{L}$. We
    must also show that $\outcomeR(G+H)=\mathscr{R}$, which simply follows from
    \cref{lem:left-end-plus-N}.
\end{proof}

\begin{theorem}
    \label{cor:B-pfree-closed}
    The set of $\mathscr{P}$-free blocking forms $\pf{B}$ is closed under
    addition. (And hence $\pf{B}[B]$ is, too.)
\end{theorem}

\begin{proof}
    This follows immediately from \cref{thm:p-free-closed} upon observing that
    the hypothesis is achieved via
    \cref{lem:B-outcome-stable,lem:B-int-invertible,lem:B-prop-X} and the
    observation that $\set{B}$ is hereditary (it is a universe).
\end{proof}

\subsection{Invertibility}

As we have mentioned previously, Milley and Renault used the ideas of tipping
points in order to characterise the invertible elements of the dead-ending
universe. With our generalised theory, we can now do similarly. We first give
results for the blocking universe specifically, after which we discuss the many
consequences.

Given a monoid of games $\set{A}$, we follow Davies and Yadav
\cite[p.~8]{davies.yadav:invertibility} in writing $\set{A}^\times$ and
$\set{A}^{\overline{\times}}$ for the set of $\set{A}$-invertible and conjugate
$\set{A}$-invertible elements of $\set{A}$ respectively. Since
$(\set{A},\equiv_\set{A})$ is a monoid, we remark that
$(\set{A}^\times,\equiv_\set{A})$ is a subgroup (the set of invertible elements
of a monoid is always a subgroup of the monoid). Given groups $\set{G}$ and
$\set{H}$, recall the notation $\set{H} \leq \set{G}$ used to mean that
$\set{H}$ is a subgroup of $\set{G}$. It was also shown in \cite[Proposition
3.4 on p.~9]{davies.yadav:invertibility} that
$\set{A}^{\overline{\times}}\leq\set{A}^\times$, and also that
$\set{U}^{\overline{\times}}=\set{U}^\times$ for every universe $\set{U}$
\cite[Theorem 3.7 on p.~11]{davies.yadav:invertibility}.

Recall that \cref{thm:B-pfree-strong} tells us that if $G\in\pf{B}$ and
$\outcome(G)\neq\mathscr{R}$, then $G$ is Left $\set{B}$-strong. Note that, by
symmetry, if we have $G\in\pf{B}$ with $\outcome(G)=\mathscr{N}$, then we must
have that $G$ is both Left and Right $\set{B}$-strong. For a symmetric form
(i.e.\ a form $G$ with $G=\overline{G}$), the only possible outcomes are
$\mathscr{N}$ and $\mathscr{P}$. But, the crucial consequence (and indeed
motivation) for proving that $\pf{B}$ is closed under addition is the
implication that $G+\overline{G}$ (which is necessarily a symmetric form) must
have outcome $\mathscr{N}$ for all $G\in\pf{B}$.

\begin{lemma}
    \label{lem:final-piece}
    If $G\in\pf{B}[B]$, then $G+\overline{G}$ is Left $\set{B}$-strong.
\end{lemma}

\begin{proof}
    We may assume $G\in\pf{B}$, from which it must follow also that
    $\overline{G}\in\pf{B}$, and hence $G+\overline{G}\in\pf{B}$ by
    \cref{cor:B-pfree-closed}. Since $G+\overline{G}$ is a symmetric form, we
    know that is must have outcome $\mathscr{N}$ or $\mathscr{P}$; but
    $G+\overline{G}\in\pf{B}$, and so $\outcome(G+\overline{G})=\mathscr{N}$.
    Thus, we may conclude by \cref{thm:B-pfree-strong} that $G+\overline{G}$ is
    Left $\set{B}$-strong.
\end{proof}

\begin{proposition}
    \label{lem:p-free-invertible}
    If $\set{U}$ is a universe such that $\pf{U}$ is additively closed, and
    $G+\overline{G}$ is Left $\set{U}$-strong for all $G\in\pf{U}$, then
    $\pf{U}\leq\set{U}^\times$ (and similarly for $\pf{U}[U]$).
\end{proposition}

\begin{proof}
    Let $G\in\pf{U}$. Following \cref{thm:comparison}, we have the proviso by
    hypothesis (i.e.\ $G+\overline{G}$ is Left and Right $\set{U}$-strong), and
    then we have the maintenance by induction. Thus, $G\in\set{U}^\times$, and
    we have the result since $\pf{U}$ is additively closed by hypothesis.
\end{proof}

We could have alternatively used \cite[Theorem 3.8 on
p.~13]{davies.yadav:invertibility} to prove \cref{lem:p-free-invertible}, but
that would have introduced some unnecessary complexity.

\begin{theorem}
    \label{prop:b-pfree-invert}
    If $G\in\pf{B}[B]$, then $G$ is $\set{B}$-invertible. That is,
    $\pf{B}[B]\leq\set{B}^\times$.
\end{theorem}

\begin{proof}
    This follows immediately from \cref{lem:p-free-invertible} upon observing
    that the hypothesis is satisfied via
    \cref{cor:B-pfree-closed,lem:final-piece}.
\end{proof}

Using this result, Davies and Milley proved in \cite{davies.milley:order} that
a game $G\in\set{B}$ is $\set{B}$-invertible if and only if its
$\set{B}$-simplest form is $\mathscr{P}$-free---an almost identical result to
the dead-ending case. Equivalently, $\pf{B}[B]=\set{B}^\times$, just as
$\pf{E}[E]=\set{E}^\times$ \cite{milley.renault:invertible}.

To finish, we prove a few (simple) general results about submonoids of monoids
whose $\mathscr{P}$-free games are invertible. Given our result for the
blocking universe (in particular, \cref{prop:b-pfree-invert}), the following
results will yield insights about a large number of submonoids. Just like we
mentioned at the end of \cref{sec:combo} (where we gave an explicit example),
we have to consider the case of $\pf{A}=\emptyset$ due to there being no
guarantee that the identity of a monoid is $\mathscr{P}$-free.

\begin{proposition}
    If $\set{A}$ is a submonoid of a monoid $\set{S}$ such that
    $\pf{S}\leq\set{S}^\times$ and, for all $G\in\pf{A}$ and all $J\in\pf{S}$,
    \begin{quote}
        if $G+J\equiv_\set{S}0$, then there exists some $H\in\pf{A}$ with
        $H\equiv_\set{A}J$,
    \end{quote}
    then either $\pf{A}\leq\set{A}^\times$ or else $\pf{A}=\emptyset$ (and
    similarly for $\pf{A}[A]$).
\end{proposition}

\begin{proof}
    Assume $\pf{A}\neq\emptyset$. Since $\set{A}\subseteq\set{S}$, it follows
    that $\pf{A}\subseteq\pf{S}$. Note that $\pf{A}=\set{A}\cap\pf{S}$, and so
    $\pf{A}$ must be additively closed. Now let $G\in\pf{A}$. It remains to
    show $G\in\set{A}^\times$.

    Since $G\in\pf{A}$, we must have also that $G\in\pf{S}$. Because
    $\pf{S}\leq\set{S}^\times$ by hypothesis, it follows that there exists some
    $J\in\pf{S}$ such that $G+J\equiv_\set{S}0$. But $\set{A}\subseteq\set{S}$,
    and so $G+J\equiv_\set{A}0$. By hypothesis, there exists some $H\in\pf{A}$
    such that $J\equiv_\set{A}H$. Thus, it is clear that $G+H\equiv_\set{A}0$,
    yielding that $G\in\set{A}^\times$, and also the result.
\end{proof}

\begin{proposition}
    \label{prop:compat-submonoid-conj}
    If $\set{A}$ is a submonoid of a monoid $\set{S}$ such that
    $\pf{S}\leq\set{S}^{\overline{\times}}$ and
    \begin{quote}
        for all $G\in\pf{A}$, there exists some $H\in\pf{A}$ with
        $H\equiv_\set{A}\overline{G}$,
    \end{quote}
    then either $\pf{A}\leq\set{A}^{\overline{\times}}$ or else
    $\pf{A}=\emptyset$ (and similarly for $\pf{A}[A]$).
\end{proposition}

\begin{proof}
    Assume $\pf{A}\neq\emptyset$ and let $G\in\pf{A}$. It suffices to show that
    $G\in\set{A}^{\overline{\times}}$. Since $\set{A}\subseteq\set{S}$, it
    follows that $\pf{A}\subseteq\pf{S}$, and so $G\in\pf{S}$. By hypothesis,
    $G\in\set{S}^{\overline{\times}}$, and so $G+\overline{G}\equiv_\set{S}0$.
    Since $\set{A}\subseteq\set{S}$, we have $G+\overline{G}\equiv_\set{A}0$.
    By hypothesis, there exists some $H\in\pf{A}$ with
    $H\equiv_\set{A}\overline{G}$, and hence we have the result.
\end{proof}

\begin{corollary}
    \label{cor:compat-submonoid}
    If $\set{A}$ is a conjugate-closed submonoid of a universe $\set{U}$ with
    $\pf{U}\leq\set{U}^\times$, then either
    $\pf{A}\leq\set{A}^{\overline{\times}}$ or else $\pf{A}=\emptyset$ (and
    similarly for $\pf{A}[A]$).
\end{corollary}

\begin{proof}
    Assume $\pf{A}\neq\emptyset$ and let $G\in\pf{A}$. Since $\set{A}$ is
    conjugate-closed, it follows that $\overline{G}\in\pf{A}$. We know that
    $\set{U}$ has the conjugate property (it is a universe), and so
    $\set{U}^\times=\set{U}^{\overline{\times}}$. Hence, we have the result by
    \cref{prop:compat-submonoid-conj}.
\end{proof}

Since the blocking universe satisfies $\pf{B}[B]\leq\set{B}^\times$ (indeed, it
is the largest such known universe), \cref{cor:compat-submonoid} tells us that,
for every conjugate-closed submonoid $\set{A}$ of $\set{B}$, we must have
$\pf{A}[A]\leq\set{A}^{\overline{\times}}$ or else $\pf{A}[A]=\emptyset$. For
example, the set of $\mathscr{P}$-free \textsc{domineering} positions
\emph{must} be a subgroup of the conjugate-invertible \textsc{domineering}
subgroup. Of course, the dead-ending universe satisfying
$\pf{E}[E]\leq\set{E}^\times$ would also have implied this result about
\textsc{domineering}, but we can now say things about rulesets that are not
dead-ending (but still blocking), like \textsc{maze} and \textsc{cricket pitch}
(see \cite[pp.~306, 311]{albert.nowakowski.ea:lessons}) .

\section{Final remarks}
\label{sec:final-remarks}

It should be clear at this point that the set of $\mathscr{P}$-free forms is
important and warrants further study. Doing so will allow us to infer
properties of the group structure of $\set{U}^\times$ for various universes
$\set{U}$, and, excitingly, also for various conjugate-closed monoids of games.
We have many questions to ask.

\begin{problem}
    What are the maximal sets of $\mathscr{P}$-free games that are closed under
    addition? Is it possible to have a universe $\set{U}$ where $\pf{U}[U]$ is
    a monoid, but where $\pf{U}[U]\not\leq\set{U}^\times$?
\end{problem}

\begin{problem}
    What are the monoids of games that are maximal with respect to being
    outcome-stable, hereditary, integer-invertible, and having
    \hyperlink{prop:X}{Property X}? Clearly $\set{M}$ cannot be such a monoid,
    since it is not integer-invertible. But such a monoid must exist by Zorn's
    Lemma, since $\set{B}$ satisfies all of those properties.
\end{problem}

The first relevant group one might want to examine is $\set{D}^\times$, whose
elements have been characterised by Fisher, Nowakowski, and Santos
\cite[Theorem 12 on p.~7]{fisher.nowakowski.ea:invertible}. What do our results
here say for the dicot universe? (We know that $\set{D}$ is conjugate-closed,
and it is trivially a subuniverse of $\set{B}$.) Well, nothing, actually: the
only $\mathscr{P}$-free game in $\set{D}$ is 0 (this is because $*$, which has
outcome $\mathscr{P}$, is the only dicot born on day 1, and hence it must be a
subposition of every non-zero dicot). Thus, although it must follow from our
results that $\pf{D}[D]\leq \set{D}^\times$, it turns out that $\pf{D}[D]$ is
the trivial subgroup. And indeed, for all universes $\set{U}$, it holds that
$\pf{U}[U]$ is the trivial group if and only if $1\notin\set{U}$; i.e.\ if and
only if $\set{D}(\overline{1})\not\subseteq\set{U}$. Recall that
$\set{D}(\set{A})$ refers to the universal closure of $\set{A}$: the smallest
universe containing a set of games $\set{A}$ (see
\cite[pp.~195--196]{siegel:on}).

So, to apply our theory, we will want to analyse those universes that are
superuniverses of $\set{D}(\overline{1})$---which is a lot of universes!
Specifically, we will want to look at those universes $\set{U}$ that satisfy
$\set{D}(\overline{1})\subseteq\set{U}\subseteq\set{B}$. And, we can
immediately point out that $\pf{U}[U]$ is non-trivial since $\overline{1}$ must
be invertible; in fact, $\pf{U}[U]$ must contain a subgroup isomorphic to
$\mathbb{Z}$. As an initial point of embarkation, we pose the understanding of
$\set{D}(\overline{1})^\times$ as a challenge to the reader.

\begin{problem}
    What can we say about the group structure of
    $\pf*{\set{D}(\overline{1})}[\set{D}(\overline{1})]$, particularly in
    relation to $\set{D}(\overline{1})^\times$? It may be worthwhile
    investigating the subgroups generated by elements born by day $n$.
\end{problem}

\begin{problem}
    What are the universes (or monoids) $\set{U}$ that are maximal with respect
    to satisfying $\pf{U}[U]\leq\set{U}^\times$? Certainly $\set{M}$ is not
    such a universe, but one must exist by Zorn's Lemma, since
    $\pf{B}[B]\leq\set{B}^\times$. Is $\set{B}$ one such universe?
\end{problem}

Why is it that $\pf{E}[E]=\set{E}^\times$ and $\pf{B}[B]=\set{B}^\times$? Is
each of these some kind of Goldilocks universe whose structure happens to be
`just right'? It is important to remember that only four universes have been
studied in the context of invertibility: $\set{D}$, $\set{E}$, $\set{B}$, and
$\set{M}$ (yes, we will run out of letters eventually). Of these four
universes, two of them have this marvellous property, and the other two
emphatically do not. This is a rather high success rate. Given that there is an
uncountable number of universes, it seems reasonable to conjecture that there
should be many more exhibiting this property. A contrarian, however, might
argue that the simplicity in the descriptions of $\set{E}$ and $\set{B}$ (i.e.\
why they have been studied at all in the first place) does indeed provide an
amount of magic that is `just right'. We leave it to the reader to form their
own opinions, and we look forward to reading concrete developments on either
side of the fence.

\begin{problem}[Goldilocks problem]
    If $\set{A}$ is a conjugate-closed monoid that lies within a universe
    $\set{U}$ with $\pf{U}[U]\leq\set{U}^\times$, and $\pf{A}[A]\neq\emptyset$,
    then we know that $\pf{A}[A]\leq\set{A}^{\overline{\times}}$. But when is
    $\pf{A}[A]=\set{A}^{\overline{\times}}$?
\end{problem}

Most of the magic comes from this result: if $G\in\set{E}$ has outcome
$\mathscr{P}$, then $G+\overline{G}$ is \emph{not} Left $\set{E}$-strong. The
analogous result holds for $\set{B}$, too. The only wrinkle is that one also
has to show that, if the $\set{U}$-simplest form of a game in $\set{U}$ is
$\mathscr{P}$-free, then there exists some equivalent game in $\set{U}$ that is
also $\mathscr{P}$-free. This happens to be true for both $\set{E}$ and
$\set{B}$: in $\set{E}$, taking the `canonical form' is enough; in $\set{B}$,
the situation is slightly more complicated, but involves modifying a `canonical
form'. The crux of the issue in general is that you have to be able to replace
end-reversible options with $\mathscr{P}$-free end-reversible options. If this
is always true, then you are happy. If not, then perhaps the definition of
being $\mathscr{P}$-free modulo $\set{U}$ would need to be altered to allow for
equivalence with augmented forms that admit $\set{U}$-expansions. This would be
unfortunate in that it is not so obvious how such a definition would generalise
for monoids that are not universes.

\bibliographystyle{plainurl}
\bibliography{bib}

\newpage

\appendix

\section{Table summaries}
\label{app:tables}

\renewcommand{\arraystretch}{1.2}
\setlength{\tabcolsep}{4pt}
\begin{table}[h]
\caption{Summary from \cref{sec:tp-basics}}
\vspace{3pt}
\centering
\begin{tabular}{l|lr}
    \hline
    \rowcolor{gray!50} Reference & \multicolumn{2}{c}{Statement} \\[2pt]
    \hline
                                 & \multicolumn{2}{c}{If
                                     $\set{A}\subseteq\maug$ is an
                                 outcome-stable, hereditary}\\
                                 & \multicolumn{2}{c}{monoid of games
                                     containing $1$ and
                                 $\overline{1}$,}\\
                                 & \multicolumn{2}{c}{and
                                 $G\in\pf{A}$,}\\
    \cmidrule(lr){2-3}
    \cref{lem:tech1} & \multirow{2}{*}{if $\outcome(G)=\mathscr{L}$} &
    $\forall G^L \in \mathscr{L}, \ntp(G^L)\leq\rtp(G)$ \\
    \cref{lem:tech3}.2 & & $\forall G^R \in \mathscr{L}\cup\mathscr{N},
    \rtp(G^R)\geq\ntp(G)$\\
    \cref{lem:tech3}.3 & & $\exists \, G^R \mid\rtp(G^R)=\ntp(G)$ \\
    \cmidrule(lr){3-3}
    \cref{lem:tech3}.1 & \multicolumn{1}{r}{$\land$if
    $\ntp(G)\neq\rtp(G)-1$}
                       &
    $\exists
    \, G^L \in \mathscr{L}\mid \ntp(G^L)=\rtp(G)$\\
    \cmidrule(lr){1-3}
    \cref{lem:tech1} & \multirow{4}{*}{if $\outcome(G)=\mathscr{N}$} &
    $\forall G^L \in \mathscr{L}, \ntp(G^L)\leq\rtp(G)$ \\
    \cref{lem:tech1}{\footnotesize \it ~sym}& & $\forall G^R \in \mathscr{R},
    \ntp(G^R)\leq\hspace{2pt}\ltp(G)$\\
    \cref{lem:tech2} & & $\exists \, G^L \in
    \mathscr{L}\mid\ntp(G^L)=\rtp(G)$\\
    \cref{lem:tech2}{\footnotesize \it ~sym}& & $\exists \, G^R \in 
    \mathscr{R}\mid
    \ntp(G^R)=\hspace{2pt}\ltp(G)$\\
    \cmidrule(lr){1-3}
    \cref{lem:tech1}{\footnotesize \it ~sym}& \multirow{2}{*}{if 
    $\outcome(G)=\mathscr{R}$} &
    $\forall G^R \in \mathscr{R}, \ntp(G^R)\leq\hspace{2pt}\ltp(G)$ \\
    \cref{lem:tech3}.2{\footnotesize \it ~sym}& & $\forall G^L \in 
    \mathscr{N}\cup\mathscr{R},
    \ltp(G^L)\geq\ntp(G)$\\
    \cref{lem:tech3}.3{\footnotesize \it ~sym}& & $\exists \, G^L \mid 
    \ltp(G^L)=\ntp(G)$ \\
    \cmidrule(lr){3-3}
    \cref{lem:tech3}.1{\footnotesize \it ~sym}& \multicolumn{1}{r}{$\land$if 
    $\ntp(G)\neq\ltp(G)-1$}
    &
    $\exists \, G^R \in \mathscr{R}\mid \ntp(G^R)=\hspace{2pt}\ltp(G)$\\
    \cmidrule(lr){1-3}
    \multirow{2}{*}{\cref{prop:ends-tp}} & if $G$ a Left end &
    $\rtp(G)=\ntp(G)+1\;$ \\
    & if $G$ a Right end & $\ltp(G)=\ntp(G)+1$. \\
    \hline
    \hline
\end{tabular}
\label{tab:techs}
\end{table}

\renewcommand{\arraystretch}{1.2}
\begin{table}
\caption{Summary of lemmas from \cref{sec:int}}
\vspace{3pt}
\centering
\begin{tabular}{r|lrl}
    \hline
    \rowcolor{gray!50} Reference & \multicolumn{3}{c}{Statement} \\[2pt]
    \hline
                                 & \multicolumn{3}{c}{If
                                     $\set{A}\subseteq\maug$ is an
                                 outcome-stable and}\\
                                 & \multicolumn{3}{c}{integer-invertible monoid
                                     of games, and
                                 $G,H\in\pf{A}[A]$,}\\
    \cmidrule(lr){2-4}
    \multirow{2}{*}{\cref{lem:11}} & \multirow{2}{*}{if
    $\outcome(G)=\mathscr{L}$ and}
    & then if $\ntp(G)>\ltp(H)$, & $\outcome(G+H)=\mathscr{L}$ \\
    \cmidrule(lr){3-4}
    & $~~\outcome(H)=\mathscr{N}$, & else if 
    $\rtp(G)<\ltp(H)$, &
    $\outcome(G+H)=\mathscr{N}$ \\
    \cmidrule(lr){1-4}
    \multirow{2.5}{*}{\cref{lem:11}}& \multirow{2}{*}{if
    $\outcome(G)=\mathscr{R}$ and}
    & then if $\ntp(G)>\rtp(H)$, & $\outcome(G+H)=\mathscr{R}$ \\
    \cmidrule(lr){3-4}
    {\footnotesize \it ~sym} & $~~\outcome(H)=\mathscr{N}$, & 
    else if 
    $\ltp(G)<\rtp(H)$, &
    $\outcome(G+H)=\mathscr{N}$ \\
    \cmidrule(lr){1-4}
    \multirow{4}{*}{\cref{lem:13}} & \multirow{4}{*}{if
    $\outcome(G),\outcome(H)=\mathscr{N}$}
    & then if $\rtp(G)>\ltp(H)$ &
    \multirow{2}{*}{$\outcome(G+H)\geq\mathscr{N}$} \\
    & & or $\ltp(G)<\rtp(H)$, & \\
    \cmidrule(lr){3-4}
    & & else if $\rtp(G)<\ltp(H)$ &
    \multirow{2}{*}{$\outcome(G+H)\leq\mathscr{N}$} \\
    & & or $\ltp(G)>\rtp(H)$, & \\
    \cmidrule(lr){1-4}
    \multirow{8}{*}{\cref{lem:14}} & & then if $\ntp(G)>\ltp(H)$, &
    $\outcome(G+H)=\mathscr{L}$ \\
    \cmidrule(lr){3-4}
    & & then if $\ntp(G)>\ntp(H)$ &
    \multirow{2}{*}{$\outcome(G+H)\geq\mathscr{N}$}\\
    &  & or $\rtp(G)>\ltp(H)$, &\\
    \cmidrule(lr){3-4}
    &  & then if $\ntp(G)>\ntp(H)$ &  \\
    &  & and $\rtp(G)<\ltp(H)$ & \multirow{2}{*}{$\outcome(G+H)=\mathscr{N}$}
    \\
    & \multirow{2}{*}{if $\outcome(G)=\mathscr{L}$ and} & or if
    $\ntp(G)<\ntp(H)$ &
    \\
    & \multirow{2}{*}{$~~~\outcome(H)=\mathscr{R}$} & and $\rtp(G)>\ltp(H)$, &
    \\
    \cmidrule(l){1-1} \cmidrule(lr){3-4}
    \multirow{3}{*}{\cref{lem:14}} & & then if $\ntp(G)<\ntp(H)$ &
    \multirow{2}{*}{$\outcome(G+H)\leq\mathscr{N}$}\\
    \multirow{2.25}{*}{{\footnotesize \it ~sym}}&  & or $\rtp(G)<\ltp(H)$, &\\
    \cmidrule(lr){3-4}
    & & then if $\rtp(G)<\ntp(H)$, &
    $\outcome(G+H)=\mathscr{R}$. \\
    \hline
    \hline
\end{tabular}
\label{tab:11-14}
\end{table}

\renewcommand{\arraystretch}{1.2}
\begin{table}
\caption{Summary of \cref{sec:combo}}
\vspace{3pt}
\centering
\begin{tabular}{l|rrr}
    \hline
    \rowcolor{gray!50} Reference & \multicolumn{3}{c}{Statement}  \\[2pt]
    \hline
    & \multicolumn{3}{c}{If $\set{A}$ is an \emph{outcome-stable},
    \emph{hereditary}, and} \\
    & \multicolumn{3}{c}{integer-invertible monoid with
    \hyperlink{prop:X}{Property X}, and} \\
    & \multicolumn{3}{c}{$G,H\in\pf{A}[A]$} \\
    \cmidrule(lr){2-4}
    \multirow{8}{*}{\cref{thm:final-piece}} &
    \multirow{2}{*}{if $\outcome(G)=\mathscr{L}$ and} & then if
    $\ntp(G)=\ltp(H)$, & 
    $\outcome(G+H)=\,\mathscr{L}$ \\
    \cmidrule(lr){3-4}
    & \multirow{1}{*}{$\outcome(H)=\mathscr{N}$,} & or if $\rtp(G)=\ltp(H)$,
    &
    $\outcome(G+H)=\mathscr{N}$ \\
    \cmidrule(lr){2-4}
    & \multirow{2}{*}{if $\outcome(G)=\mathscr{R}$ and} & then if
    $\ntp(G)=\rtp(H)$,
    & $\outcome(G+H)=\;\mathscr{R}$ \\
    \cmidrule(lr){3-4}
    & \multirow{1}{*}{$\outcome(H)=\mathscr{N}$,} & or if $\ltp(G)=\rtp(H)$,
    &
    $\outcome(G+H)=\mathscr{N}$ \\
    \cmidrule{2-4}
    & \multirow{1.3}{*}{if $\outcome(G)=\mathscr{N}$ and} & 
    \multirow{1.3}{*}{and if $\rtp(G)=\ltp(H)$} &
    \multirow{2}{*}{$\outcome(G+H)=\mathscr{N}$} \\
    & \multirow{1}{*}{$\outcome(H)=\mathscr{N}$,} & 
    \multirow{1}{*}{or $\ltp(G)=\rtp(H)$,} & \\
    \cmidrule{2-4}
    &  & if $\ntp(G)=\ltp(H)$, & 
    $\outcome(G+H)=\,\mathscr{L}$ \\
    \cmidrule(lr){3-4}
    & \multirow{1.5}{*}{if $\outcome(G)=\mathscr{L}$ and} & if
    $\ntp(G)=\ntp(H)$  &
    \multirow{2}{*}{$\outcome(G+H)=\mathscr{N}$} \\
    & \multirow{-1}{*}{$\outcome(H)=\mathscr{R}$,} & or $\rtp(G)=\ltp(H)$, & \\
    \cmidrule(lr){3-4}
    &  & if $\rtp(G)=\ntp(H)$, & $\outcome(G+H)=\;\mathscr{R}$ \\

    \hline
    \hline
\end{tabular}
\label{tab:combo}
\end{table}

\end{document}